\documentclass[english]{smfart}
\usepackage{sabbah_twisted-derham2}
\begin{document}
\frontmatter
\title{On a twisted de~Rham complex, II}

\author[C.~Sabbah]{Claude Sabbah}
\address{UMR 7640 du CNRS\\
Centre de Mathématiques Laurent Schwartz\\
École polytechnique\\
F--91128 Palaiseau cedex\\
France}
\email{sabbah@math.polytechnique.fr}
\urladdr{http://www.math.polytechnique.fr/~sabbah}

\thanks{This research was supported by the grant ANR-08-BLAN-0317-01 of the Agence nationale de la recherche.}

\begin{abstract}
We prove an algebraic formula, conjectured by M. Kontsevich, for computing the monodromy of the vanishing cycles of a regular function on a smooth complex algebraic variety.
\end{abstract}



\maketitle

\mainmatter

\section{Introduction}\label{sec:intro}

Let $X$ be a smooth complex algebraic variety equipped with its Zariski topology and let $f:X\to\Afu$ be a function on $X$ (\ie $f\in\Gamma(X,\cO_X)$). Let $\hb$ be a new variable. We use the following notation: given any $\CC$-vector space $E$, we denote by $E\lcr\hb\rcr$ the $\CC\lcr\hb\rcr$-module of formal power series with coefficients in $E$ and by $E\lpr\hb\rpr$ the $\CC\lpr\hb\rpr$-vector space of Laurent formal series with coefficient in $E$. For a sheaf~$\cF$ on $X$, $\cF\lpr\hb\rpr$ denotes the sheaf associated to the presheaf $U\mto\cF(U)\lpr\hb\rpr$. The sheaf $\cO_X\lpr\hb\rpr$ is $\cO_X$-flat. We denote by $\wh\cE_X^{-f/u}$ the sheaf $\cO_X\lpr\hb\rpr$ equipped with the connection $d-df/u$.

Let $\cM$ be a locally free $\cO_X$-module of finite rank equipped with a flat~$\nabla$ having regular singularity at infinity (\cf \cite{Deligne70}). Then $\wh\cE_X^{-f/u}\otimes_{\cO_X}\cM$ is a locally free $\cO_X\lpr\hb\rpr$-module equipped with the connection $\nabla-df/u\otimes\id_\cM$. Note that we have $\cO_X\lpr\hb\rpr\otimes_{\cO_X}\nobreak\cM=\cM\lpr\hb\rpr$ since $\cM$ is $\cO_X$-coherent. We will consider the formal twisted de~Rham complex
\[
\DR(\wh\cE_X^{-f/u}\otimes_{\cO_X}\cM)=\big(\Omega_X^{\cbbullet}\lpr\hb\rpr\otimes_{\cO_X}\cM,\nabla-df\otimes\id_\cM/\hb\big).
\]
It comes equipped with a $\CC\lpr\hb\rpr$-connection defined by $\nabla_{\partial_\hb}=\partial_\hb+f/\hb^2$, which commutes with the differential. The hypercohomology spaces on $X$ of this complex are therefore $\CC\lpr\hb\rpr$-vector spaces with a connection $\nabla_{\partial_\hb}$.

On the other hand, let $f^\an:X^\an\to\CC$ be the associated holomorphic morphism and let $\cL=\ker\nabla^\an$ be the local system of horizontal sections of $\nabla^\an$. For each $t_o\in\CC$, let $\phi_{f-t_o}\cL$ be the complex of vanishing cycles of $f$ along the fibre $f^{-1}(t_o)$ with coefficients in $\cL$, equipped with its monodromy operator $\rT$. We have $\phi_{f-t_o}\cL\simeq0$ if $t_o$ is not a critical value of $f$, and such critical values form a finite set in $\Afu$. The hypercohomology $\bH^k\big(f^{-1}(t_o),\phi_{f-t_o}\cL\big)$ is a finite dimensional $\CC$-vector space equipped with a monodromy operator $\rT$.

In general, let $E$ be a finite dimensional $\CC$-vector space equipped with an automorphism $\rT$. Given a choice of a logarithm of $\rT$, that is, writing $\rT=\exp(-2\pi i \rM)$ for some $\rM:E\to E$, we denote by $\whRHm(E,\rT)$ the $\CC\lpr\hb\rpr$-vector space $E\lpr\hb\rpr$ equipped with the connection $d+\rM d\hb/\hb$. Given $t_o\in\CC$, we set $\wh\cE^{-t_o/\hb}=(\CC\lpr\hb\rpr,d+t_od\hb/\hb^2)$.

\begin{theoreme}\label{th:main}
We have, for each $k$,
\bgroup\numstareq
\begin{multline}\label{eq:main*}
\Big(\bH^{k+1}\big(X,\DR(\wh\cE_X^{-f/u}\otimes_{\cO_X}\cM)\big),\nabla_{\partial_\hb}\Big)\\\simeq\bigoplus_{t_o\in\CC}\wh\cE^{-t_o/u}\otimes_{\CC\lpr\hb\rpr}\whRHm\Big(\bH^k\big(f^{-1}(t_o),\phi_{f-t_o}\cL\big),\rT\Big).
\end{multline}
\egroup
\end{theoreme}

Notice that, since $\cM$ is $\cO_X$-coherent, the left-hand term can also be written as
\[
\Big(\bH^{k+1}\big(X,(\Omega_X^{\cbbullet}\otimes\cM\lpr\hb\rpr,\nabla-df\otimes \id_\cM/\hb)\big),\nabla_{\partial_\hb}\Big),
\]
(see \cite[Prop\ptbl5.1]{Hartshorne75}) and, in this form, the result has been conjectured by M\ptbl Kontsevich.

\begin{proof}
We use the propositions stated and proved below in the following way. We first show that the natural morphism from the left-hand term of \eqref{eq:main*} to the corresponding analytic object is an isomorphism. By considering a covering of $X$ by quasi-projective Zariski open sets and a spectral sequence argument, we reduce to the case where $X$ is quasi-projective. Then, by using a compactification $F$ of $f$ as in \eqref{eq:diag}, we can replace the left-hand term of \eqref{eq:main*} with that of \eqref{eq:gaga*} below, and express it in an analytic way as the right-hand term of \eqref{eq:gaga*}, according to Proposition \ref{prop:gaga}. The assertion follows from Proposition \ref{prop:comparaison}.

The analytic analogue of the left-hand term is in turn identified with the right-hand term of \eqref{eq:main*} according to Proposition \ref{prop:BSK}, as we can use a Nagata compactification of the graph of $f$ in order to apply the constructibility results of \S\ref{subsec:hrconst}.
\end{proof}

Assume for instance that $(\cM,\nabla)=(\cO_X,d)$. It is more common (\cf \eg \cite{Bibi97b}) to consider the algebraic twisted de~Rham complex $\big(\Omega_X^{\cbbullet}[\hb,\hbm],d-df/\hb\big)$. The hypercohomology $\big(\bH^{k+\dim X}\big(X,(\Omega_X^{\cbbullet}[\hbm],d-df\hbm)\big),\nabla_{\partial_\hb}\big)$ is known to be identified with the Laplace transform $\Fou M^k$ of the $k$-th direct image $M^k:=\cH^k f_+\cO_X$ of the $\cD_X$-module $\cO_X$, and therefore $\big(\bH^{k+\dim X}\big(X,(\Omega_X^{\cbbullet}[\hb,\hbm],d-df/\hb)\big),\nabla_{\partial_\hb}\big)$ is identified with $G^k:=\CC[\hb,\hbm]\otimes_{\CC[\hbm]}\nobreak \Fou M^k$. Using the regularity of $M^k$, the compatibility of vanishing cycles with proper direct images and classical results on the Laplace transform of a regular holonomic $\Clt$-module, one finds
\begin{multline}\label{eq:polu}
\Big(\CC\lpr\hb\rpr\otimes_{\CC[\hb,\hbm]}\bH^{k+1}\big(X,(\Omega_X^{\cbbullet}[\hb,\hbm],d-df/\hb)\big),\nabla_{\partial_\hb}\Big)\\
\simeq\bigoplus_{t_o\in\CC}\wh\cE^{-t_o/\hb}\otimes_\CC\whRHm\Big(\bH^k\big(g^{-1}(t_o),\phi_{g-t_o}\bR j_*\CC_{X^\an}\big),\rT\Big),
\end{multline}
where we use a commutative diagram with $j$ open and $g$ proper:
\[
\xymatrix{
X\ar@<-.3ex>@{^{ (}->}[r]^-j\ar[rd]_f&X'\ar[d]^{g}\\&\Afu
}
\]

If $f$ is proper (so that $f=g$), the right-hand sides of \eqref{eq:polu} and \eqref{eq:main*} coincide. It is not clear a priori that the left-hand sides coincide, so we cannot obtain \eqref{eq:main*} from \eqref{eq:polu} directly when $f$ is proper, but this follows from the results explained in \S\ref{sec:andR} below, which are an easy consequence of \cite{MSaito86} and \cite{Kapranov91}. On the other hand, if $f$ is not proper, the left-hand sides of \eqref{eq:polu} and \eqref{eq:main*} may differ, as shown in the following example, and thus \eqref{eq:main*} needs a different argument.

\begin{exemple}
Let $f\in\CC[t]$ be a non-constant polynomial in one variable and let $X$ be the Zariski open set of $\Afu$ complementary to $\{f'=0\}$. Then the formal twisted de~Rham complex%
\[
\CC[t,1/f']\lpr\hb\rpr\To{\hb\partial_t-f'}\CC[t,1/f']\lpr\hb\rpr
\]
has zero cohomology. Indeed, let us show for instance that the differential is onto. This amounts to showing that, given $\psi_{k_o},\psi_{k_o+1},\dots$ in $\CC[t,1/f']$, we can find $\varphi_{k_o},\varphi_{k_o+1},\dots$ in $\CC[t,1/f']$ such that
\[
\psi_{k_o}=-f'\varphi_{k_o},\quad\psi_{k_o+1}=\partial_t\varphi_{k_o}-f'\varphi_{k_o+1},\dots,\psi_{k+1}=\partial_t\varphi_{k}-f'\varphi_{k+1},\dots,
\]
a system which can be solved inductively because $f'$ is invertible in $\CC[t,1/f']$.

On the other hand, the complex
\[
\CC[t,1/f'][\hb,\hbm]\To{\hb\partial_t-f'}\CC[t,1/f'][\hb,\hbm]
\]
has cohomology in degree one only, and this cohomology is a free $\CC[\hb,\hbm]$-module of rank equal to $\deg f\cdot\#\{f(t)\mid f'(t)=0\}$.
\end{exemple}

\subsubsection*{Acknowledgements}
I thank Morihiko Saito for pointing out various inaccuracies and a gap in a first version of this article and for providing the method to fill it up, as well as for suggesting various simplifications and improvements.

\section{Formal twisted de~Rham complexes}

\subsection{Preliminary results on $\cF\lpr\hb\rpr$}\label{subsec:prelim}
Let $Y$ be a topological space and let $\cF$ be a sheaf of $\CC$-vector spaces on $Y$. Let $\hb$ be a new variable and set $\cF[\hb]=\CC[\hb]\otimes_\CC\nobreak\cF$. We denote by $\cF\lcr\hb\rcr$ the sheaf $\varprojlim_k\cF[\hb]/\hb^k\cF[\hb]$. A germ of section of $\cF\lcr\hb\rcr$ at $y\in Y$ consists of a series $\sum_{n\geq0}f_{n,y}\hb^n$ where $f_{n,y}$ is the germ at $y\in Y$ of a section $f_n\in\Gamma(U,\cF)$ for some open neighbourhood $U$ of $y$ which does not depend on $n$. Since the coefficients of a power series are uniquely determined, we have
\[
\Gamma(U,\cF\lcr\hb\rcr)=\Gamma(U,\cF)\lcr\hb\rcr
\]
for any open set $U\subset Y$. We now have $\cF\lpr\hb\rpr=\CC\lpr\hb\rpr\otimes_{\CC\lcr\hb\rcr}\cF\lcr\hb\rcr$.

Clearly, the projective system above is surjective. Therefore, given a bounded complex of sheaves on $Y$, $\lcr\hb\rcr$ commutes with taking cohomology sheaves in the cases considered in \cite[\S4]{Hartshorne75} as well as in the case where $Y$ is a complex analytic manifold and each term of the complex is an inductive limit of coherent $\cO_Y$-modules (the functor~$\sigma$ of \loccit being here the product of spaces of sections on compact Stein polydiscs). It follows that $\lpr\hb\rpr$ also commutes with taking cohomology sheaves in these cases.

Assume now that $Y$ is a complex algebraic variety, $D$ is a divisor in~$Y$, \hbox{$j:\nobreak Y\moins D\hto Y$} is the inclusion, and $\cF$ is a coherent sheaf of $\cO_Y(*D)$-modules. The sheaf $\cO_Y(*D)\lcr\hb\rcr$ is $\cO_Y(*D)$-flat and the natural morphism
\[
\cO_Y(*D)\lcr\hb\rcr\otimes_{\cO_Y(*D)}\cF\to\cF\lcr\hb\rcr
\]
is an isomorphism, as well as the corresponding analytic one (argue as in \cite[Prop\ptbl5.1]{Hartshorne75}). Tensoring with $\CC\lpr\hb\rpr$ gives\begin{equation}\label{eq:compu}
\cO_Y(*D)\lpr\hb\rpr\otimes_{\cO_Y(*D)}\cF\isom\cF\lpr\hb\rpr.
\end{equation}

\begin{lemme}\label{lem:gagau}
If $Y$ is projective, the algebraic/analytic comparison morphism
\[
H^k(Y,\cF\lpr\hb\rpr)\to H^k(Y^\an,\cF^\an\lpr\hb\rpr)
\]
is an isomorphism for each $k\in\NN$.
\end{lemme}

\begin{proof}
Since~$\cF$ is an inductive limit of coherent $\cO_Y$-modules, the algebraic/analytic comparison theorem holds for $\cF$ (\cf \cite[\S II.6.5]{Deligne70}). Similarly, it holds for each sheaf $\cF_n\defin\cF[\hb]/\hb^n\cF[\hb]$.

We first claim that the algebraic/analytic comparison morphism
\begin{equation}\label{eq:lcrhb}
H^k(Y,\cF\lcr\hb\rcr)\to H^k(Y^\an,\cF^\an\lcr\hb\rcr)
\end{equation}
is an isomorphism for each $k\in\NN$. Indeed, the same method as in \cite[Prop\ptbl6.1]{Hartshorne75} can be applied. One has to check that the conditions of \cite[Th\ptbl4.5]{Hartshorne75} are fulfilled for the projective systems $(\cF_n)$ and $(\cF^\an_n)$. In the algebraic case, we consider the basis of affine open sets of $Y$. Then, for each such set $U$, the projective system $\Gamma(U,\cF_n)=\Gamma(U\moins D,j^*\cF_n)$ is surjective, and $H^k(U,\cF_n)=0$ for $k>0$ and any~$n$. In the analytic case, we consider instead the family of compact polydiscs in $Y^\an$ (with respect to any choice of local coordinate system). For $U$ in such a family, $H^k(U,\cF^\an_n)=0$ for $k>0$ and any $n$, since $\cF^\an_n$ is an inductive limit of coherent sheaves. The argument of \cite[Th\ptbl4.5]{Hartshorne75} applies similarly to such a family. It remains to check that both projective systems $H^k(Y,\cF_n)$ and $H^k(Y^\an,\cF^\an_n)$ satisfy the Mittag-Leffler condition, a property which is clear since $\cF_n$ is a direct summand of $\cF_{n+1}$.

Since $Y$ is Noetherian and $Y^\an$ compact, and since $\CC\lpr\hb\rpr$ is $\CC\lcr\hb\rcr$-flat, we have $H^k(Y,\cF\lpr\hb\rpr)=\CC\lpr\hb\rpr\otimes_{\CC\lcr\hb\rcr}H^k(Y,\cF\lcr\hb\rcr)$ (\cf\cite[Prop\ptbl III.2.9]{Hartshorne80}, because $\cF\lpr\hb\rpr=\varinjlim_n\hb^{-n}\otimes\cF\lcr\hb\rcr$) and a similar equality for $Y^\an$ (\cf\cite[Prop\ptbl2.6.6]{K-S90}). Hence, tensoring \eqref{eq:lcrhb} with $\CC\lpr\hb\rpr$ gives the result 
\end{proof}

Let us end this section by comparing the effect of $\CC\lpr\hb\rpr\otimes_{\CC\lcr\hb\rcr}$ before and after taking direct image by $f$. As already remarked, both functors commute if $f$ is proper, or in the algebraic setting by working with the Zariski topology. We will now consider the case of an open embedding in the analytic topology.

\begin{proposition}\label{prop:compdirlim}
Let $Y$ be a complex manifold, let $Z\subset Y$ be a closed analytic subset of~$Y$ and let $j:X:=Y\moins Z\hto Y$ denote the open inclusion. Let $\cF^\cbbullet$ be a bounded complex of sheaves of $\CC\lcr\hb\rcr$-modules on $X$. Assume that $\cF^\cbbullet$ has constructible cohomology with respect to a Whitney stratification induced from one on $Y$. Then the natural morphism
\[
\CC\lpr\hb\rpr\otimes_{\CC\lcr\hb\rcr}\bR j_*\cF^\cbbullet\to\bR j_*(\CC\lpr\hb\rpr\otimes_{\CC\lcr\hb\rcr}\cF^\cbbullet)
\]
is a quasi-isomorphism.
\end{proposition}

\begin{proof}
Let $\bD_{\CC\lcr\hb\rcr}$ (\resp $\bD_{\CC\lpr\hb\rpr}$) denotes the Poincaré-Verdier duality functor on the derived category of complexes of sheaves of $\CC\lcr\hb\rcr$-modules (\resp $\CC\lpr\hb\rpr$-modules). By using flatness of $\CC\lpr\hb\rpr$ over $\CC\lcr\hb\rcr$, one checks that the natural morphism $\bD_{\CC\lpr\hb\rpr}(\CC\lpr\hb\rpr\otimes_{\CC\lcr\hb\rcr}\cF^\cbbullet)\to\CC\lpr\hb\rpr\otimes_{\CC\lcr\hb\rcr}\bD_{\CC\lcr\hb\rcr}(\cF^\cbbullet)$ is an isomorphism in the suitable derived category. Now, the proposition clearly holds if we replace $\bR j_*$ with $\bR j_!$, and moreover $\bR j_!\cF^\cbbullet$ has constructible cohomology on $Y$. Since Poincaré-Verdier duality exchanges both functors (\cf\eg\cite[Ex\ptbl VIII.3]{K-S90}), the proposition holds for $\bR j_*$.
\end{proof}

\begin{proof}[Sketch of another proof, due to M\ptbl Saito]
Fix a Whitney stratification of $Y$ such that~$Z$ is a union of strata, and denote by $X_k$ the union of $X$ and of the strata of codimension $\leq k$ in $Z$. We have the open inclusions $X\Hto{j_0}X_0\Hto{j_1}X_1\cdots\hto Y$, whose composition is $j$. By induction, it is enough to prove the result for the successive inclusions $j_k:X_{k-1}\hto X_k$ instead of $j$. The set $Z_k=X_k\moins X_{k-1}$ is a smooth manifold (the union of strata of codimension $k$ in $Z$). In the neighbourhood of each point $x_o$ of $Z_k$, $X_k$ is homeomorphic to $V\times C(L)$, where $V$ is a neighbourhood of $x_o$ in~$Z_k$, and $C(L)$ is the cone over the link $L$ of $V$ in $X_k$. In this picture, the inclusion~$j_k$ is induced by the inclusion $V\times C(L)^*\hto V\times C(L)$, with $C(L)^*=C(L)\moins\text{apex}$. The complex~$\cF^\cbbullet$ is then isomorphic to the pull-back by $\rho:V\times C(L)^*\to V\times L$ of a constructible complex $\cG^\cbbullet$ on $V\times L$, and $\iota^{-1}\bR j_*\cF^\cbbullet$ is identified with $\bR p_*\cG^\bbullet$, where $p:V\times L\to V$ denotes the projection. Since $p$ is proper, $\CC\lpr\hb\rpr\otimes_{\CC\lcr\hb\rcr}$ commutes with $\bR p_*$, hence the result.
\end{proof}

\subsection{Algebraic-analytic comparison for the formal twisted de~Rham complex}\label{subsec:gaga}
We now come back to the setting of \S\ref{sec:intro}. Let $F:Y\to\PP^1$ be a morphism from a smooth projective variety $Y$ to~$\PP^1$, extending $f$, that is, such that there is a commutative diagram
\begin{equation}\label{eq:diag}
\begin{array}{c}
\xymatrix{
X\ar[d]_f\ar@{^{ (}->}[r]^-j&Y\ar[d]^F\\\CC\ar@{^{ (}->}[r]&\PP^1
}
\end{array}
\end{equation}
and such that $Y\moins X$ is a divisor $D$ in $Y$. Then $(\cM,\nabla)$ extends as a coherent $\cO_Y(*D)$-module with connection having regular singularity along $D$, that we continue to denote by $(\cM,\nabla)$. In particular $\cM$ is $\cD_Y$-holonomic, hence $\cD_Y$-coherent. We will denote by $\wh\cO_{(Y,D)}$ the sheaf $\cO_Y(*D)\lpr\hb\rpr$ and $\wh\cD_{(Y,D)}=\wh\cO_{(Y,D)}\otimes_{\cO_Y}\nobreak\cD_Y=\cD_Y\otimes_{\cO_Y}\nobreak\wh\cO_{(Y,D)}$. We denote by $\wh\cE_{(Y,D)}^{-F/u}$ the $\wh\cD_{(Y,D)}$-module $\cO_Y(*D)\lpr\hb\rpr$ equipped with the connection $d-dF/u$. The left-hand term in \eqref{eq:main*} becomes (up to changing~$k+1$ to $k$)
\begin{multline*}
\Big(\bH^k\big(Y,\DR(\wh\cE_{(Y,D)}^{-F/u}\otimes_{\cO_Y}\cM)\big),\nabla_{\partial_\hb}\Big)\\=\Big(\bH^k\big(Y,(\Omega_Y^{\cbbullet}(*D)\lpr\hb\rpr\otimes\cM,\nabla-dF/\hb\otimes \id_\cM)\big),\nabla_{\partial_\hb}\Big).
\end{multline*}
Similarly, we denote by $Y^\an$, $\cM^\an$, etc., the corresponding analytic objects, and we consider the analytic twisted de~Rham cohomology.

\begin{proposition}\label{prop:gaga}
Let $\cM$ be a coherent $\cD_Y$-module. Then, for each $k\geq0$, the natural morphism
\bgroup\numstareq
\begin{equation}\label{eq:gaga*}
\bH^k\big(Y,\DR(\wh\cE_{(Y,D)}^{-F/u}\otimes_{\cO_Y}\cM)\big)
\to \bH^k\big(Y^\an,\DR(\wh\cE_{(Y,D)}^{-F/u}\otimes_{\cO_{Y^\an}}\cM^\an)\big)
\end{equation}
\egroup
is an isomorphism.
\end{proposition}

\begin{proof}[Proof]
The natural morphism \eqref{eq:gaga*} comes from a morphism of complexes, and it is enough to show that the morphism between the $E_1$ terms
\[
H^k(Y,\Omega_Y^\ell(*D)\lpr\hb\rpr\otimes_{\cO_Y}\cM)\to H^k(Y^\an,\Omega_{Y^\an}^\ell(*D)\lpr\hb\rpr\otimes_{\cO_{Y^\an}}\cM^\an)
\]
of the natural spectral sequences is an isomorphism (\cf \cite[\S II.6.6]{Deligne70}). We know that $\cM=\varinjlim_n\cM_n$ with $\cM_n$ coherent over~$\cO_Y$. Applying the preliminary results above to $Y$ and $\cF=\cO_Y(*D)\otimes_{\cO_Y}\cM_n$, we conclude that this comparison morphism is an isomorphism for each $\cM_n$ and thus, since $Y$ is Noetherian and $Y^\an$ is compact, for~$\cM$ also by passing to the inductive limit.
\end{proof}

\subsection{Direct images of formal twisted de~Rham complexes}
The results in this subsection have been suggested by M\ptbl Saito. We will work in the category of smooth algebraic (or complex analytic) varieties~$Y$ equipped with a reduced divisor~$D$. Morphisms $\pi:(Y',D')\to(Y,D)$ consist of morphisms $\pi:Y'\to\nobreak Y$ such that $\pi^{-1}(D)=D'$. In particular, $\wh\cO_{(Y',D')}$ is a $\pi^{-1}\wh\cO_{(Y,D)}$-module. Let~$F$ be a rational (or meromorphic) function on $Y$ with poles along $D$. 

It will be convenient to denote by $\pDR(\cbbullet)$ the complex $\DR(\cbbullet)$ shifted by the dimension of the ambient variety.

\begin{proposition}\label{prop:imdiru}
If $\cM'$ is $\cD_{Y'}$-coherent (or good relative to $\pi$ in the complex analytic case) and $\pi$ is proper, there is a functorial isomorphism
\[
\pDR(\wh\cE_{(Y,D)}^{-F/u}\otimes_{\cO_Y}\pi_+\cM')\isom\bR\pi_*\pDR(\wh\cE_{(Y',D')}^{-(F\circ\pi)/u}\otimes_{\cO_{Y'}}\cM').
\]
\end{proposition}

\begin{remarque}
If $\cM$ is $\cO_Y(*D)$-coherent, it follows from \eqref{eq:compu} that the natural morphism
\[
\DR(\wh\cE_{(Y,D)}^{-F/u}\otimes_{\cO_Y}\cM)\to\big(\Omega_Y^\cbbullet\otimes\cM\lpr\hb\rpr,\nabla-df/\hb\otimes\id_{\cM})
\]
is an isomorphism (termwise) in the algebraic or the analytic setting.

However, in the setting of Proposition \ref{prop:imdiru}, if $\cM'$ is $\cO_{Y'}(*D')$-coherent but $\pi_+\cM'$ is not with $\cO_Y(*D)$-coherent cohomology, this does not apply to $\cH^j\pi_+\cM'$ and the complex $\bR\pi_*(\Omega_{Y'}^{n'+\cbbullet}\otimes\nobreak\cM'\lpr\hb\rpr,\nabla'-d(F\circ\pi)\otimes\id_{\cM'}/\hb)$ may not be quasi-isomorphic to $(\Omega_Y^{n+\cbbullet}\otimes\nobreak(\pi_+\cM')\lpr\hb\rpr,\nabla-dF\otimes\id_{\pi_+\cM'}/\hb)$, already for a finite morphism $\pi$ (\eg \hbox{$\pi:t'\mto t=t^{\prime2}$} from $\Afu$ to $\Afu$, $F(t)=t$, $D=\emptyset$ and $\cM'=(\cO_{\Afu},d)$, an example due to M\ptbl Saito). 
\end{remarque}

\begin{proof}[\proofname\ of Proposition \ref{prop:imdiru}]
Let us set
\[
\wh\cD_{(Y,D)\from(Y',D')}=\pi^{-1}\wh\cD_{(Y,D)}\otimes_{\pi^{-1}\cD_Y}\cD_{Y\from Y'}\otimes_{\cD_{Y'}}\wh\cD_{(Y',D')},
\]
which is a $(\pi^{-1}\wh\cD_{(Y,D)},\wh\cD_{(Y',D')})$-bimodule. If $\wh\cM'$ is a left $\wh\cD_{(Y',D')}$-module, we set (as usual) $\wh\pi_+\wh\cM'=\bR\pi_*(\wh\cD_{(Y,D)\from(Y',D')}\otimes^{\bL}_{\wh\cD_{(Y',D')}}\wh\cM')$. The proposition is a direct consequence of Lemma \ref{lem:imdirfu} below, by using the isomorphism $\bR\pi_*\pDR\simeq\pDR\,\wh\pi_+$ proved in the present setting as for $\cD$-modules. 
\end{proof}

\begin{lemme}\label{lem:imdirfu}
With the assumptions of Proposition \ref{prop:imdiru}, there is a functorial isomorphism
\[
\wh\cE_{(Y,D)}^{-F/u}\otimes_{\cO_Y}\pi_+\cM'\to\wh\pi_+(\wh\cE_{(Y',D')}^{-(F\circ\pi)/u}\otimes_{\cO_{Y'}}\cM').
\]
\end{lemme}

\begin{proof}
The proof is done in the following steps.
\begin{enumerate}
\item
We prove that there is a functorial isomorphism in $D^b(\pi^{-1}\wh\cD_{(Y,D)})$:
\begin{multline*}
\wh\cD_{(Y,D)\from(Y',D')}\otimes^{\bL}_{\wh\cD_{(Y',D')}}(\wh\cE_{(Y',D')}^{-(F\circ\pi)/u}\otimes_{\cO_{Y'}}\cM')\\
\simeq
\pi^{-1}\wh\cE_{(Y,D)}^{-F/u}\otimes_{\wh\cO_{(Y,D)}}\wh\cD_{(Y,D)\from(Y',D')}\otimes^{\bL}_{\wh\cD_{(Y',D')}}(\wh\cO_{(Y',D')}\otimes_{\cO_{Y'}}\cM').
\end{multline*}
Applying $\bR\pi_*$ we find a morphism in $D^b(\wh\cD_{(Y,D)})$
\[
\wh\pi_+(\wh\cE_{(Y',D')}^{-(F\circ\pi)/u}\otimes_{\cO_{Y'}}\cM')\to\wh\cE_{(Y,D)}^{-F/u}\otimes_{\wh\cO_{(Y,D)}}\wh\pi_+(\wh\cO_{(Y',D')}\otimes_{\cO_{Y'}}\cM'),
\]
which is functorial with respect to $\cM'$, and is an isomorphism if $\pi$ is proper.
\item
It is then enough to prove that the natural morphism in $D^b(\wh\cD_{(Y,D)})$
\[
\wh\cO_{(Y,D)}\otimes_{\cO_Y}\pi_+\cM'\to\wh\pi_+(\wh\cO_{(Y',D')}\otimes_{\cO_{Y'}}\cM')
\]
is an isomorphism. We will argue as in \cite[Prop\ptbl6.2]{Hartshorne75}.
\end{enumerate}

For these steps, it is equivalent but simpler to give the proof of the analogous statements for a right $\cD$-module $\cM'$. We then set $\wh\cM'=\cM'\otimes_{\cO_{Y'}}\wh\cO_{(Y',D')}$.

\subsubsection*{Proof of Step one}
Let us note that $\wh\cM'\otimes_{\wh\cO_{(Y',D')}}\wh\cE_{(Y',D')}^{-(F\circ\pi)/u}$ is naturally equipped with a right $\wh\cD_{(Y',D')}$-module structure. One first checks that, for any left $\wh\cD_{(Y',D')}$-module $\wh\cM''$, there is a natural isomorphism
\[
\big(\wh\cM'\otimes_{\wh\cO_{(Y',D')}}\wh\cE_{(Y',D')}^{-(F\circ\pi)/u}\big)\otimes_{\wh\cD_{(Y',D')}}\wh\cM''\isom\wh\cM'\otimes_{\wh\cD_{(Y',D')}}\big(\wh\cE_{(Y',D')}^{-(F\circ\pi)/u}\otimes_{\wh\cO_{(Y',D')}}\wh\cM''\big)
\]
given by $(m'\otimes1)\otimes m''\mto m'\otimes(1\otimes m'')$.

Applying this to $\wh\cM''=\Sp^\cbbullet(\wh\cD_{(Y',D')})\otimes_{\wh\cO_{(Y',D')}}\wh\cD_{(Y',D')\to(Y,D)}$, where $\Sp^\cbbullet$ is the Spencer complex, we find
\begin{multline*}
\big(\wh\cM'\otimes_{\wh\cO_{(Y',D')}}\wh\cE_{(Y',D')}^{-(F\circ\pi)/u}\big)\otimes^{\bL}_{\wh\cD_{(Y',D')}}\wh\cD_{(Y',D')\to(Y,D)}\\
\simeq\wh\cM'\otimes^{\bL}_{\wh\cD_{(Y',D')}}\big(\wh\cE_{(Y',D')}^{-(F\circ\pi)/u}\otimes_{\wh\cO_{(Y',D')}}\wh\cD_{(Y',D')\to(Y,D)}\big),
\end{multline*}
as right $\pi^{-1}\wh\cD_{(Y,D)}$-modules. Now, as $(\wh\cD_{(Y',D')},\pi^{-1}\wh\cD_{(Y,D)})$-bimodules,
\begin{align*}
\wh\cE_{(Y',D')}^{-(F\circ\pi)/u}\otimes_{\wh\cO_{(Y',D')}}\wh\cD_{(Y',D')\to(Y,D)}&=\wh\cE_{(Y',D')}^{-(F\circ\pi)/u}\otimes_{\pi^{-1}\wh\cO_{(Y,D)}}\pi^{-1}\wh\cD_{(Y,D)}\\
&=\wh\cO_{(Y',D')}\otimes_{\pi^{-1}\wh\cO_{(Y,D)}}\pi^{-1}(\wh\cE_{(Y,D)}^{-F/u}\otimes_{\wh\cO_{(Y,D)}}\wh\cD_{(Y,D)}).
\end{align*}
It remains to check that, as a left and right $\wh\cD_{(Y,D)}$-module, $\wh\cE_{(Y,D)}^{-F/u}\otimes_{\wh\cO_{(Y,D)}}\wh\cD_{(Y,D)}$ is isomorphic to $\wh\cD_{(Y,D)}\otimes_{\wh\cO_{(Y,D)}}\wh\cE_{(Y,D)}^{-F/u}$. This is obtained by sending $1\otimes P(x,\hb,\partial)$ to $P(x,\hb,\partial+\hb^{-1}\partial F)\otimes 1$.\qed

\subsubsection*{Proof of Step two}
By functoriality, it is enough to consider the case where $\cM'=\cF'\otimes_{\cO_{Y'}}\cD_{Y'}$, where $\cF'$ is $\cO_{Y'}$-coherent. We now notice that $(\cF'\otimes_{\cO_{Y'}}\cD_{Y'})\otimes_{\cO_{Y'}}\wh\cO_{(Y',D')}=(\wh\cO_{(Y',D')}\otimes_{\cO_{Y'}}\cF')\otimes_{\wh\cO_{(Y',D')}}\wh\cD_{(Y',D')}$. We will make use of the following lemmas.

\begin{lemme}\label{lem:imdiru}
Let $\pi:Y'\to Y$ be a proper morphism of smooth complex algebraic varieties, let $D'$ be a divisor in $Y'$, and let $\cG'$ be a coherent $\cO_{Y'}(*D')$-module. Then, for each $k$, the natural morphism $R^k\pi_*(\cG'\lpr\hb\rpr)\to R^k\pi_*(\cG')\lpr\hb\rpr$ is an isomorphism. The same result holds in the analytic category, provided that $\cG'$ is good (\ie $\cG'=\cO_{Y'}(*D')\otimes_{\cO_{Y'}}\cF'$ for some $\cO_{Y'}$-coherent submodule $\cF'\subset\cG'$).
\end{lemme}

\begin{proof}
We apply \cite[Th\ptbl4.5]{Hartshorne75} to the functor $\pi_*$, and argue as in Lemma \ref{lem:gagau}, by working first with $\lcr\hb\rcr$.
\end{proof}

\begin{lemme}\label{lem:imdiru2}
With the same assumption as in Lemma \ref{lem:imdiru}, assume moreover that $D'=\pi^{-1}(D)$ for some divisor $D$ of $Y$. Then, for each $k$, the natural morphism $\cO_Y(*D)\lpr\hb\rpr\otimes_{\cO_Y(*D)}R^k\pi_*(\cG')\to R^k\pi_*\big(\cO_{Y'}(*D')\lpr\hb\rpr\otimes_{\cO_{Y'}(*D')}\cG'\big)$ is an isomorphism.
\end{lemme}

\begin{proof}
A priori, we cannot apply the projection formula, since the natural morphism $\cO_{Y'}(*D')\otimes_{\pi^{-1}\cO_Y(*D)}\pi^{-1}\cO_Y(*D)\lpr\hb\rpr\to\cO_{Y'}(*D')\lpr\hb\rpr$ is not an isomorphism in general. But the result follows from Lemma \ref{lem:imdiru} together with \eqref{eq:compu}, since $R^k\pi_*(\cG')$ is $\cO_Y(*D)$-coherent (indeed, $\cG'=\cO_{Y'}(*D')\otimes_{\cO_{Y'}}\cF'$ for some $\cO_{Y'}$-coherent submodule $\cF'\subset\cG'$; we also have $\cG'=\pi^{-1}\cO_Y(*D)\otimes_{\pi^{-1}\cO_Y}\cF'$ and we can then apply the projection formula).
\end{proof}

We conclude Step two by noticing that $\pi_+(\cF'\otimes_{\cO_{Y'}}\cD_{Y'})=(\bR\pi_*\cF')\otimes_{\cO_Y}\cD_Y$ and
\[
\wh\pi_+\big((\wh\cO_{(Y',D')}\otimes_{\cO_{Y'}}\cF')\otimes_{\wh\cO_{(Y',D')}}\wh\cD_{(Y',D')}\big)=\bR\pi_*(\wh\cO_{(Y',D')}\otimes_{\cO_{Y'}}\cF')\otimes_{\wh\cO_{(Y,D)}}\wh\cD_{(Y,D)},
\]
so the desired isomorphism follows from Lemma \ref{lem:imdiru2}:
\begin{align*}
\pi_+(\cF'\otimes_{\cO_{Y'}}\cD_{Y'})\otimes_{\cO_Y}\wh\cO_{(Y,D)}&=(\bR\pi_*\cF')\otimes_{\cO_Y}\cD_Y\otimes_{\cO_Y}\wh\cO_{(Y,D)}\\
&=(\bR\pi_*\cF')\otimes_{\cO_Y}\wh\cO_{(Y,D)}\otimes_{\cO_Y}\cD_Y\\
&=(\bR\pi_*\cF')\otimes_{\cO_Y}\wh\cO_{(Y,D)}\otimes_{\wh\cO_{(Y,D)}}\wh\cD_{(Y,D)}\\
&=\bR\pi_*(\cF'\otimes_{\cO_{Y'}}\wh\cO_{(Y',D')})\otimes_{\wh\cO_{(Y,D)}}\wh\cD_{(Y,D)}\\
&=\wh\pi_+\big((\wh\cO_{(Y',D')}\otimes_{\cO_{Y'}}\cF')\otimes_{\wh\cO_{(Y',D')}}\wh\cD_{(Y',D')}\big).\qedhere
\end{align*}
\end{proof}

\section{Formal twisted de~Rham complex in the analytic setting and nearby/vanishing cycles}\label{sec:andR}
In this section, $X$ denotes a complex analytic manifold of dimension~$n$, and \hbox{$f:X\to\CC$} is a holomorphic function on $X$. We denote by $\cM$ a locally free $\cO_X$-module of finite rank $d$ equipped with a flat holomorphic connection $\nabla$, whose associated local system $\ker\nabla$ is denoted by $\cL$.

We will improve here \cite[Th\ptbl2.4(a)]{Kapranov91} in the sense that we compute the analytic twisted de~Rham complex in terms of the vanishing cycle complex without grading with respect to some filtration, by using some results in \cite{MSaito86}. This will produce the local analytic comparison analogous to \eqref{eq:main*}. Most ingredients in the proof of the following propositions are already present in \cite{MSaito86, Kapranov91,MSaito94,B-S07}, and the new input only consists in putting everything together, with some details. Moreover, we show how similar arguments can be used to obtain global comparison results, which is our main objective.

\subsection{Complexes of $\cD_{\CC,0}$-modules with constructible cohomology}\label{subsec:hrconst}
Let $Y$ be a complex analytic set and let $\cF$ be a sheaf of left $\cD_{\CC,0}$-modules, where $\cD_{\CC,0}=\CC\{t\}\langle\partial_t\rangle$ is the Noetherian ring of differential operators with coefficients in $\CC\{t\}$. We say that~$\cF$ is \emph{constructible} if it is locally constant on each stratum of some complex analytic stratification of $Y$, and for each $y\in Y$, $\cF_y$ is a $\cD_{\CC,0}$-module of finite type. We will be mainly concerned with the case where each $\cF_y$ is a regular holonomic $\cD_{\CC,0}$-module, in which case we say that $\cF$ is a \emph{constructible sheaf of regular holonomic $\cD_{\CC,0}$-modules}. We denote by $V^\cbbullet\cD_{\CC,0}$ the decreasing filtration indexed by $\ZZ$ such that $V^0\cD_{\CC,0}=\CC\{t\}\langle t\partial_t\rangle$, $V^k\cD_{\CC,0}=t^kV^0\cD_{\CC,0}$ and $V^{-k}\cD_{\CC,0}=\sum_{j=0}^k\partial_t^kV^0\cD_{\CC,0}$ for $k\geq0$.

Let $\cF$ be such a sheaf, and let $\cY=(Y_a)_{a\in A}$ be a stratification with respect to which it is constructible. Then $\cF$ comes equipped with a canonical (decreasing) filtration indexed by $\ZZ$, called the \emph{Kashiwara-Malgrange filtration}, that we denote by $V^\cbbullet\cF$, characterized by the following properties:
\begin{itemize}
\item
$V^k\cF=t^kV^0\cF$ if $k\geq0$,
\item
$V^{-k-1}\cF=V^{-1}\cF+\cdots+\partial_t^kV^{-1}\cF$ if $k\geq0$,
\item
for each $a\in A$, each $V^k\cF_{|Y_a}$ is a locally constant sheaf of $\CC\{t\}$-modules of finite type, which is stable under the action of $t\partial_t$ (it is enough to check this property for $V^{-1}\cF_{|Y_a}$, according to the previous point),
\item
for each $a\in A$, each $\gr_V^k\cF_{|Y_a}$ is a locally constant sheaf of finite dimensional $\CC$-vector spaces and, locally on $Y$, there exists a nonzero polynomial $b(s)\in\CC[s]$ with roots having their real part in $[0,1)$ such that $b(t\partial_t-k)$ vanishes on $\gr_V^k\cF$.
\end{itemize}

We denote by $D^b_{\Cc}(Y,\cD_{\CC,0})$ (\resp $D^b_{\Cc,\hr}(Y,\cD_{\CC,0})$ the bounded derived category of sheaves of $\cD_{\CC,0}$-modules with constructible cohomology (\resp regular holonomic constructible cohomology).

\begin{proposition}[\cf proof of Prop\ptbl3.4 in \cite{Kapranov91}]\label{prop:Dconst}
Let $\cF$ be an object of $D^b_{\Cc,\hr}(Y,\cD_{\CC,0})$ and let $g:Y\to Z$ be a holomorphic map between complex analytic set. Assume that $\Supp\cF^\cbbullet$ is an analytic Zariski open set in a an analytic space $\ov Y$ and that $g_{|\Supp\cF^\cbbullet}$ extends as a proper map $g_{\ov Y}:\ov Y\to Z$. Then $\bR g_*\cF$ is an object of $D^b_{\Cc,\hr}(Z,\cD_{\CC,0})$.
\end{proposition}

The assumption on $\Supp\cF^\cbbullet$ means that $\ov Y\moins\Supp\cF^\cbbullet$ is a closed analytic subset of~$\ov Y$. Proposition \ref{prop:Dconst} does not immediately follow from standard results on construc\-tible complexes since $\cD_{\CC,0}$ is not commutative. We first start with a more precise result on constructible sheaves.

\begin{lemme}\label{lem:Dconst}
Let $\cF$ be a regular holonomic constructible sheaf on $Y$ and assume $g$ extends to a proper map as in Proposition \ref{prop:Dconst}. Then
\begin{enumerate}
\item\label{lem:Dconst1}
each $R^ig_*\cF$ is a regular holonomic constructible sheaf on $Z$,
\item\label{lem:Dconst2}
for each $i,k$, each $R^ig_*V^k\cF$ is a constructible sheaf of $\CC\{t\}$-modules on $Z$,
\item\label{lem:Dconst3}
for each $i,k$, the natural morphism $R^ig_*V^k\cF\to R^ig_*\cF$ is an inclusion, whose image is equal to $V^kR^ig_*\cF$.
\end{enumerate}
\end{lemme}

\begin{proof}
We can apply standard results on constructible sheaves of $A$-modules, with $A=\CC\{t\}$, which has finite global homological dimension, being local and regular (\cf \cite[Th\ptbl9, p\ptbl103]{Serre65}). Then \ref{lem:Dconst}\eqref{lem:Dconst2} follows from \cite[Prop\ptbl8.5.7(b)]{K-S90}: indeed, if $j:\Supp\cF^\cbbullet\hto\nobreak \ov Y$ denotes the inclusion, then $\bR j_*V^k\cF^\cbbullet$ has $\CC\{t\}$-constructible cohomology, since clearly $\bR j_!V^k\cF^\cbbullet$ has so and since duality preserves constructibility; then $\bR g_*V^k\cF^\cbbullet\simeq\bR g_{\ov Y,*}\bR j_*V^k\cF^\cbbullet$ also has $\CC\{t\}$-constructible cohomology. In order to get \ref{lem:Dconst}\eqref{lem:Dconst3} we will restrict to the pull-back of a compact set~$K$ of~$Z$ and we will apply the following lemma.

\begin{lemme}[{\cite[\S3.3]{MSaito86}, see also \cite[Prop\ptbl3.1.13]{Bibi01c}}]\label{lem:imdirmono}
Let $(\cN^\cbbullet,U^\cbbullet\cN^\cbbullet)$ be a $V$-filtered complex of sheaves of $\cD_{\CC,0}$-modules on $Y$ (\ie $V^\ell\cD_{\CC,0}\cdot U^k\cN^i\subset U^{k+\ell}\cN^i$ for all $i,k,\ell$), with the following properties:
\begin{enumerate}
\item\label{lem:imdirmono1}
locally on $Y$, there exists a nonzero polynomial $b(s)\in\CC[s]$, with roots having real part in $[0,1)$, such that $b(t\partial_t-k)$ vanishes on $\gr_U^k\cN^i$ for each $i$ and $k$,
\item\label{lem:imdirmono2}
there exists $k_0$ such that, for all $k\geq k_0$ and all $i$, the left multiplication by $t$ induces an isomorphism $t:U^k\cN^i\to U^{k+1}\cN^i$ of $\CC\{t\}$-modules,
\item\label{lem:imdirmono3}
there exists $k_0$ such that, for all $k\geq k_0$, all $i$ and all $y\in Y$, the $t$-torsion submodule $\cT^{i,k}_y\subset\cH^i(U^k\cN^\cbbullet)_y$ satisfies $\cT^{i,k}_y\cap t^\ell\cH^i(U^k\cN^\cbbullet)_y=0$ for some $\ell\geq0$,
\item\label{lem:imdirmono4}
there exists $i_0$ such that, for all $i\geq i_0$ and any $k$, one has $\cH^i(U^k\cN^\cbbullet)=0$.
\end{enumerate}
Then, for any $i,k$, the natural morphism $\cH^i(U^k\cN^\cbbullet)\to\cH^i(\cN^\cbbullet)$ is injective and its image defines a filtration $U^k\cH^i(\cN^\cbbullet)$ satisfying $\gr_U^k\cH^i(\cN^\cbbullet)=\cH^i(\gr_U^k\cN^\cbbullet)$.\qed
\end{lemme}

We notice that the properties of the lemma hold for $(\cF,V^\cbbullet\cF)_{|g^{-1}(K)}$: indeed, \ref{lem:imdirmono}\eqref{lem:imdirmono2} holds with $k_0=0$, and this implies that \ref{lem:imdirmono}\eqref{lem:imdirmono3} also holds for $k_0=0$ and $\ell=0$. To obtain \ref{lem:imdirmono}\eqref{lem:imdirmono1}, let us note that $\gr^k_V\cF$ is constructible with respect to the same stratification as $\cF$ is, and that the minimal polynomial $b$ of the action of $t\partial_t$ is constant on the connected strata. By properness of $g_{\ov Y}$, \ref{lem:imdirmono}\eqref{lem:imdirmono1} holds on $g^{-1}(K)$.

Let $\God^\cbbullet\cF$ be the standard Godement flabby resolution of $\cF$, which is filtered by $\God^\cbbullet V^k\cF$ (\cf\cite[p\ptbl167]{Godement64}). We claim that the properties of the lemma also hold for $(g_*\God^\cbbullet\cF,g_*\God^\cbbullet V^\cbbullet\cF)$. This is clear for \ref{lem:imdirmono}\eqref{lem:imdirmono2}, \ref{lem:imdirmono}\eqref{lem:imdirmono1} follows from the exactness of $\God^\bbullet$ (\cf\cite[p\ptbl168]{Godement64}) and \ref{lem:imdirmono}\eqref{lem:imdirmono1} for $(\cF,V^\cbbullet\cF)$, and \ref{lem:imdirmono}\eqref{lem:imdirmono4} from the finite homological dimension of $g$. We also notice that \ref{lem:imdirmono}\eqref{lem:imdirmono3} holds if $\cH^i(U^k\cN^\cbbullet)$ is locally of finite type over $\CC\{t\}$. Since we already know that $R^ig_*V^k\cF=\cH^i(g_*\God^\cbbullet V^k\cF)$ has finite type, we obtain \ref{lem:imdirmono}\eqref{lem:imdirmono3}.

\enlargethispage{\baselineskip}%
Taking the cohomology of $(g_*\God^\cbbullet\cF,g_*\God^\cbbullet V^\cbbullet\cF)$ and applying the conclusion Lemma \ref{lem:imdirmono} gives \ref{lem:Dconst}\eqref{lem:Dconst3}.

In order to get \ref{lem:Dconst}\eqref{lem:Dconst1}, we first claim that
\[
R^ig_*\cF=\varinjlim_kR^ig_*V^k\cF=\bigcup_kR^ig_*V^k\cF.
\]
If $g$ is proper, this is obtained by applying the conclusion of Lemma \ref{lem:imdirmono} and using that $\bigcup_kV^k\cF=\cF$. If $g$ extends to a proper map $g_{\ov Y}$ as in Proposition \ref{prop:Dconst}, we set $g=g_{\ov Y}\circ j$, and argue for $Rj_*$ as in the second proof of Proposition \ref{prop:compdirlim}.

Moreover, $R^ig_*V^{-k-1}\cF=R^ig_*V^{-1}\cF+\cdots+\partial_t^kR^ig_*V^{-1}\cF$, where the sum is taken in $R^ig_*\cF$, and $R^ig_*V^{-1}\cF$ has finite type over $\CC\{t\}$, hence $R^ig_*\cF$ has finite type over $\cD_{\CC,0}$. Let us now check that there exists a complex analytic stratification $\cZ$ of $Z$ such that $R^ig_*V^k\cF$ is $\cZ$-constructible for each $k$. Let us choose a stratification $\cZ$ such that $R^ig_*V^k\cF$ is $\cZ$-constructible for $k=-1,0$. Then, for $k\geq0$, $t^k:R^ig_*V^0\cF\to R^ig_*V^k\cF$ is an isomorphism, and so $R^ig_*V^k\cF$ is $\cZ$-constructible. Similarly, $R^ig_*\gr_V^{-1}\cF=\gr_V^{-1}R^ig_*\cF$ is $\cZ$-constructible, and so is $R^ig_*\gr_V^k\cF=\gr_V^kR^ig_*\cF$ for each $k\leq-1$ since $\partial_t^{-k-1}:\gr_V^{-1}R^ig_*\cF\to\gr_V^kR^ig_*\cF$ is an isomorphism. Since the category of local systems of $\CC\{t\}$-modules is stable by extensions in the category of sheaves of $\CC\{t\}$-modules and stable by inductive limits, so is the category of $\cZ$-constructible sheaves of $\CC\{t\}$-modules, and we conclude that $V^kR^ig_*\cF=R^ig_*V^k\cF$ ($k\leq-1$) are $\cZ$-constructible, as well as their inductive limit $R^ig_*\cF$, hence \ref{lem:Dconst}\eqref{lem:Dconst1}.
\end{proof}

\begin{proof}[Proof of Proposition \ref{prop:Dconst}]
One reduces to the case of a regular holonomic constructible sheaf by using the Leray spectral sequence for $g$ and the fact that any morphism between such sheaves has kernel and cokernel in the same category. The result follows then from Lemma \ref{lem:Dconst}\eqref{lem:Dconst1}.
\end{proof}

\begin{corollaire}\label{cor:Dconst}
Let $(\cF^\cbbullet,U^\cbbullet\cF^\cbbullet)$ be a bounded $V$-filtered complex satisfying the following assumptions:
\begin{enumerate}
\item\label{cor:Dconst1}
$\cF^\cbbullet$ is an object of $D^b_{\Cc,\hr}(Y,\cD_{\CC,0})$,
\item\label{cor:Dconst2}
$t:U^k\cF^i\to U^{k+1}\cF^i$ is an isomorphism of $\CC\{t\}$-modules for each $k\geq0$ and each $i$,
\item\label{cor:Dconst3}
$U^{-k-1}\cF^i=\sum_{\ell=0}^k\partial_t^\ell U^{-1}\cF^i$ for each $k\geq0$ and each $i$,
\item\label{cor:Dconst4}
there exists a nonzero polynomial $b(s)\in\CC[s]$ with roots having their real part in $[0,1)$ such that $b(t\partial_t-k)$ vanishes on each $\gr_U^k\cF^i$ for each~$i$,
\item\label{cor:Dconst5}
for each $i$, each germ $\cH^i(U^{-1}\cF^\cbbullet)_y$ has finite type over $\CC\{t\}$.
\end{enumerate}
Then for each $i$ and $k$, the natural morphism $\cH^i(U^k\cF^\cbbullet)\to\cH^i(\cF^\cbbullet)$ is injective and has image equal to the Kashiwara-Malgrange filtration $V^k\cH^i(\cF^\cbbullet)$.

Moreover, if $\Supp\cF^\cbbullet$ is an analytic Zariski open set in a compact analytic space then each $\bH^i(Y,\cF^\cbbullet)$ is a regular holonomic $\cD_{\CC,0}$-module and, for each $i$ and $k$, the natural morphism $\bH^i(Y,U^k\cF^\cbbullet)\to\bH^i(Y,\cF^\cbbullet)$ is injective and has image equal to the Kashiwara-Malgrange filtration $V^k\bH^i(Y,\cF^\cbbullet)$.
\end{corollaire}

\begin{proof}
Let us check that $(\cF^\cbbullet,U^\cbbullet\cF^\cbbullet)$ satisfies the assumptions of Lemma \ref{lem:imdirmono}. Clearly, \ref{cor:Dconst}\eqref{cor:Dconst4} implies \ref{lem:imdirmono}\eqref{lem:imdirmono1}, \ref{cor:Dconst}\eqref{cor:Dconst2} implies \ref{lem:imdirmono}\eqref{lem:imdirmono2} and \ref{lem:imdirmono}\eqref{lem:imdirmono4} follows from the boundedness of $\cF^\cbbullet$. Lastly, \ref{cor:Dconst}\eqref{cor:Dconst5} implies \ref{lem:imdirmono}\eqref{lem:imdirmono3} with $k_0=-1$.

As a consequence of Lemma \ref{lem:imdirmono},
\[
\big(\cH^i(\gr_U^k\cF^\cbbullet),\rT=\exp(-2\pi it\partial_t)\big)=\big(\gr_U^k\cH^i(\cF^\cbbullet),\rT=\exp(-2\pi it\partial_t)\big)
\]
for each $i$ and $k$, where $U^\cbbullet\cH^i(\cF^\cbbullet)$ is the naturally induced filtration. Now, \ref{cor:Dconst}\eqref{cor:Dconst1} implies that $\cH^i(\cF^\cbbullet)$ is a regular holonomic constructible sheaf, \ref{cor:Dconst}\eqref{cor:Dconst2} and \eqref{cor:Dconst3} also hold for the induced filtration $U^\cbbullet\cH^i(\cF^\cbbullet)$, and \ref{cor:Dconst}\eqref{cor:Dconst4} and \eqref{cor:Dconst5} imply that, for each $y\in\nobreak Y$, the germ at $y$ of this filtration coincides with the Kashiwara-Malgrange filtration $V^\cbbullet\cH^i(\cF^\cbbullet)_y$. This gives the first assertion. Moreover, $\cH^i(U^k\cF^\cbbullet)$ is then a constructible sheaf of $\CC\{t\}$-modules for each $i$ and $k$, since $V^k\cH^i(\cF^\cbbullet)$ is so.

Similarly, the complex $\Gamma\big(Y,\God^\cbbullet(\cF^\cbbullet,U^\cbbullet\cF^\cbbullet)\big)$ of global sections canonical Godement resolution $\God^\cbbullet(\cF^\cbbullet,U^\cbbullet\cF^\cbbullet)$ also satisfies the assumptions of Corollary \ref{cor:Dconst}: indeed, \ref{cor:Dconst}\eqref{cor:Dconst1} follows from Proposition \ref{prop:Dconst}; \ref{cor:Dconst}\eqref{cor:Dconst5} follows from the $\CC\{t\}$-constructibility of $U^k\cF^\cbbullet$ which has been noticed above; and the other properties are clear. It follows from Proposition \ref{prop:Dconst} applied to the constant map~$g$ that $\bH^i(Y,\cF^\cbbullet)$ is a regular holonomic $\cD_{\CC,0}$-module, and by an argument similar to the one used for $\cH^i(\cF^\cbbullet)$, that $U^\cbbullet\bH^i(Y,\cF^\cbbullet)$ is its Kashiwara-Malgrange filtration.
\end{proof}

Let us denote by $\wt\cD_{\CC,0}$ the ring $\cD_{\CC,0}[1/t]=\CC(\{t\})\langle\partial_t\rangle$, where $\CC(\{t\})$ denotes the field of convergent Laurent series. Then $\wt\cD_{\CC,0}$ contains $\cD_{\CC,0}$ as a subring and is $\cD_{\CC,0}$-flat. We define $\wtRHm$ as $\whRHm$ in \S\ref{sec:intro}, by using $\CC(\{t\})$ instead of $\CC\lpr\hb\rpr$.

\begin{corollaire}\label{cor:Dconsttilde}
With the same assumptions as in Corollary \ref{cor:Dconst}, let us set $\wt\cF^\cbbullet=\wt\cD_{\CC,0}\otimes_{\cD_{\CC,0}}^{\bL}\cF^\cbbullet$. Then, for all $i$,
\bgroup\numstareq
\begin{equation}\label{eq:Dconsttilde*}
\cH^i(\wt\cF^\cbbullet)\simeq\wtRHm\big(\cH^i(\gr_U^0\cF^\cbbullet),\rT=\exp(-2\pi it\partial_t)\big),
\end{equation}
\egroup
and, if $\Supp\cF^\cbbullet$ is Zariski open in a compact analytic space,
\bgroup\numstarstareq
\begin{equation}\label{eq:Dconsttilde**}
\bH^i(Y,\wt\cF^\cbbullet)\simeq\wtRHm\big(\bH^i(Y,\gr_U^0\cF^\cbbullet),\rT=\exp(-2\pi it\partial_t)\big).
\end{equation}
\egroup
\end{corollaire}

\begin{proof}
By $\cD_{\CC,0}$-flatness of $\wt\cD_{\CC,0}$, we have $\cH^i(\wt\cF^\cbbullet)=\wt{\cH^i(\cF^\cbbullet)}$ (\resp $\bH^i(Y,\wt\cF^\cbbullet)=\wt{\bH^i(Y,\cF^\cbbullet)}$). By Corollary \ref{cor:Dconst}, we have $\cH^i(\gr_U^0\cF^\cbbullet)=\gr_V^0\cH^i(\cF^\cbbullet)$ (resp.\ $\bH^i(Y,\gr_U^0\cF^\cbbullet)=\gr_V^0\bH^i(Y,\cF^\cbbullet)$. It is therefore enough to show that, given a regular holonomic $\cD_{\CC,0}$-module $M$, we have $\wt M\simeq\wtRHm(\gr^0_VM,\exp(-2\pi it\partial_t))$, which is well-known.
\end{proof}

Let us denote by $\wh\cE_{\CC,0}$ the ring of germs at $(0;1)$ of formal micro-differential operators (\cf \eg \cite[Chap\ptbl7]{Kashiwara03}). This is ring of formal Laurent series $\sum_{k\leq k_0}a_k(t)\partial_t^k=\sum_{k\leq k_0}\partial_t^kb_k(t)$, where~$a_k,b_k$ are holomorphic in some disc $|t|\leq\epsilon$ of radius $\epsilon$ independent of $k$. The product structure is described in \loccit and $\cD_{\CC,0}$ is a subring of $\wh\cE_{\CC,0}$, making the latter a flat (left and right) module over the former. Recall that the functor $\whRHm$ has been defined in \S\ref{sec:intro}.

\begin{corollaire}\label{cor:Dconsthat}
With the same assumptions as in Corollary \ref{cor:Dconst}, let us set $\wh\cF^\cbbullet=\wh\cE_{\CC,0}\otimes_{\cD_{\CC,0}}^{\bL}\cF^\cbbullet$. Then, for all $i$,
\bgroup\numstareq
\begin{equation}\label{eq:Dconsthat*}
\cH^i(\wh\cF^\cbbullet)\simeq\whRHm\big(\cH^i(\gr_U^{-1}\cF^\cbbullet),\rT=\exp(-2\pi it\partial_t)\big),
\end{equation}
\egroup
and, if $\Supp\cF^\cbbullet$ is Zariski open in a compact analytic space,
\bgroup\numstarstareq
\begin{equation}\label{eq:Dconsthat**}
\bH^i(Y,\wh\cF^\cbbullet)\simeq\whRHm\big(\bH^i(Y,\gr_U^{-1}\cF^\cbbullet),\rT=\exp(-2\pi it\partial_t)\big).
\end{equation}
\egroup
\end{corollaire}

\begin{proof}
As for Corollary \ref{cor:Dconsttilde}, one reduces to showing the isomorphism $\wh M\simeq\whRHm(\gr^{-1}_VM,\exp(-2\pi it\partial_t))$ for a regular holonomic $\cD_{\CC,0}$-module $M$, which is also well-known.
\end{proof}

\begin{remarque}\label{rem:whF}
Notice that, for any object $\cF^\cbbullet$ of $D^b_{\Cc,\hr}(Y,\cD_{\CC,0})$, we also have a natural morphism
\begin{equation}\label{eq:morphismeformel}
\CC\lpr\partial_t^{-1}\rpr\otimes_{\CC[\partial_t]}\cF^\cbbullet\to \wh\cE_{\CC,0}\otimes_{\cD_{\CC,0}}\cF^\cbbullet=\wh\cF^\cbbullet,
\end{equation}
which is a quasi-isomorphism. Indeed, by a flatness argument, it is enough to check that, for any regular holonomic $\cD_{\CC,0}$-module $M$, a similar assertion holds. If~$M$ is supported at the origin, the result is immediate, hence it remains to check this when~$t$ acts in an invertible way on $M$, and by an extension argument, to the case $M=\cD_{\CC,0}/\cD_{\CC,0}(t\partial_t-\alpha)$ with $\reel\alpha\in[-1,0)$, for which the result is easy.

Similarly, one shows that the cohomology sheaves of $\wh\cF^\cbbullet$ are constructible sheaves of $\CC\lpr\partial_t^{-1}\rpr$-vector spaces.
\end{remarque}

\subsection{The complex $\cK_f^{\protect\cbbullet}$}
We now go back to the setting of the beginning of this section. Let $i_f:X\hto X\times\CC$ denote the graph inclusion of $f$, defined by $x\mto(x,t=\nobreak f(x))$. We consider the left $\cD_{X\times\CC}$-module $\cM_f=i_{f,+}\cM$. Firstly, since we work in the analytic setting, we have $\cM\simeq\cL\otimes_\CC\cO_X$ with the trivial connection $1\otimes d$. We then have $\cM_f\simeq\cL\otimes_\CC\cB_f$, with $\cB_f=i_{f,+}\cO_X=\cO_X[\partial_t]\delta(t-f)$ (\cf \eg \cite{MSaito94} for the notation) where the $\cD_{X\times\CC}$ action is locally defined as follows:
\begin{align*}
g(x,t)\cdot\delta(t-f)&=g(x,f(x))\delta(t-f),\\
\partial_{x_i}\cdot\delta(t-f)&=-\frac{\partial f}{\partial x_i}\,\partial_t\delta(t-f).
\end{align*}
Note that $\cM_f$ is supported on the graph of $f$ (identified with $X$ by $i_f$) and is already coherent over the ring $\cR_X\defin\cD_X[t]\langle\partial_t\rangle$.

We will consider the twisted de~Rham complex
\[
\cK_f^\cbbullet:=\DR_{X\times\CC/\CC}(\cM_f)\simeq\cL\otimes_\CC(\Omega_X^\cbbullet[\partial_t],d-df\otimes\partial_t).
\]
Since $\DR_{X\times\CC/\CC}(\cD_{X\times\CC})[n]$ is a resolution of $\cD_{\CC\from X\times\CC}$ as a $(f^{-1}\cD_\CC,\cD_{X\times\CC})$-bimodule by locally free right $\cD_{X\times\CC}$-modules, the complex $\cK_f^\cbbullet[n]$ represents $\cD_{\CC\from X\times\CC}\otimes^{\bL}_{\cD_{X\times\CC}}\nobreak\cM_f$.

If $f$ is smooth then, locally on $X$, $\cK_f^\cbbullet$ has cohomology in degree $1$ at most, and this cohomology is isomorphic to $\cL\otimes_\CC f^{-1}\cO_\CC$ as a $f^{-1}\cD_\CC$-module.

We now focus on the value $t=0$ and we set $X_0=f^{-1}(0)$. We will often forget the restriction map in the notation below. In order to consider other values $t=t_o$, one should simply replace $f$ with $f-t_o$. The sheaf of rings $\cD_{X\times\CC|X_0\times0}$ is equipped with a decreasing $V$-filtration, for which $t$ has order $1$ and $\partial_t$ has order~$-1$, and the operators in the $X$ direction have order zero. In local coordinates,
\begin{align*}
V^0(\cD_{X\times\CC})&=\Big\{\textstyle\sum_{\alpha,k}a_{\alpha,k}(x,t)\partial_x^\alpha(t\partial_t)^k\mid\alpha\in\NN^n,\;k\in\NN,\;a_{\alpha,k}\in\cO_{X\times\CC}\Big\},\\
V^\ell(\cD_{X\times\CC})&=
\begin{cases}
t^\ell V^0(\cD_{X\times\CC})&\text{if }\ell\geq0,\\
\sum_{j=0}^{-\ell}\partial_t^jV^0(\cD_{X\times\CC})&\text{if }\ell\leq0.
\end{cases}
\end{align*}

According to \cite{Kashiwara76}, there exists a unique decreasing filtration $V^\cbbullet(\cM_f)$ indexed by $\ZZ$ such that
\begin{itemize}
\item
each $V^k\cM_f$ is $V^0(\cD_{X\times\CC})$-coherent,
\item
$V^{k+\ell}\cM_f=t^\ell V^j\cM_f$ for $k\geq0$ and $\ell\geq0$,
\item
$V^{-k-\ell}\cM_f=\sum_{j=0}^\ell\partial_t^jV^\alpha\cM_f$ for $k\leq-1$ and $\ell\geq0$,
\item
locally on $X$, there exists a polynomial $b(s)\in\CC[s]$ with roots in $\QQ\cap[0,1)$ such that $b(t\partial_t-k)$ vanishes identically on $\gr_V^k\cM_f$ for each $k$.
\end{itemize}
Note that our convention is shifted by one with respect to that of \cite{MSaito94}, but corresponds to that of \cite[\S1.3]{B-S07}. The uniqueness of such a filtration is not difficult, while the existence and the fact that the roots of $b$ be chosen in $\QQ$ follows from \cite{Kashiwara76}. We call this filtration the \emph{Kashiwara-Malgrange} $V$-filtration of $\cM_f$.

The complex $\cK_f^\cbbullet$ is then filtered by setting $U^k\cK_f^\cbbullet:=\DR_{X\times\CC/\CC}(V^k\cM_f)$, which is meaningful since $V^k\cM_f$ is a $V^0\cD_{X\times\CC}$-module, hence a $\cD_{X\times\CC/\CC}$-module. This makes $(\cK_f^\cbbullet,U^\cbbullet\cK_f^\cbbullet)$ a $V$-filtered complex in the sense used in \S\ref{subsec:hrconst}. We can summarize various known properties of this $V$-filtered complex in the following theorem.

\begin{theoreme}\label{th:KM}\mbox{}
\begin{enumerate}
\item\label{th:KM1}
The complex $\gr_U^0\cK_f^\cbbullet$ (\resp $\gr_U^{-1}\cK_f^\cbbullet$) equipped with the operator $\rT_f=\exp(-2\pi it\partial_t)$ is isomorphic to the nearby cycle complex $\psi_f\cL[-1]$ (\resp $\phi_f\cL[-1]$), equipped with its monodromy.
\item\label{th:KM2}
The $V$-filtered complex $(\cK_f^\cbbullet,U^\cbbullet\cK_f^\cbbullet)$ satisfies all properties of Corollary \ref{cor:Dconst}.
\end{enumerate}
\end{theoreme}

Recall that $(\psi_f\cL,T_f)$ (\resp $(\phi_f\cL,T_f)$) denotes the nearby (\resp vanishing) cycle complex of $f$ (\cf \cite{Deligne73}).

\begin{proof}
The first point is mainly due to Malgrange \cite{Malgrange83a} and applies more generally to any regular holonomic $\cD_X$-module~$\cM$. There are various proofs or approaches of this fact (\cite{Malgrange83a, Kashiwara83,MSaito86,M-S86b,Mebkhout89,L-M95}).

In the second point, \ref{cor:Dconst}\eqref{cor:Dconst2}--\eqref{cor:Dconst4} are easily checked. On the other hand, \ref{cor:Dconst}\eqref{cor:Dconst1} is contained in \cite[Prop\ptbl3.4]{Kapranov91} and \cite[Prop\ptbl1.1]{B-S07}. Let us recall the argument of \cite{B-S07}.
\begin{enumerate}
\item
The same question can be asked for any regular holonomic $\cD_{X\times\CC}$-module $\cN$ such that $\DR_{X\times\CC/\CC}(\cN)$ has regular holonomic cohomology over $f^{-1}\cD_\CC$. Below, we will implicitly restrict sheaf-theoretically the various sheaves or complexes to $X_0$.
\item\label{enum:imdir}
Let $\pi:X'\to X$ be a proper modification with $X'$ smooth and exceptional locus contained in $X_0$, and set $f'=f\circ\pi$. We also denote by $\pi:X'\times\CC\to X\times\CC$ the induced modification. Recall that $\pi_+\cN'=\bR\pi_*(\cD_{X\times\CC\from X'\times\CC}\otimes^{\bL}_{\cD_{X'\times\CC}}\cN')$. Moreover, by using the projection formula one finds
\[
\cD_{\CC\from X\times\CC}\otimes^{\bL}_{\cD_{X\times\CC}}\pi_+\cN'\simeq\bR\pi_*(\cD_{\CC\from X'\times\CC}\otimes^{\bL}_{\cD_{X'\times\CC}}\cN'),
\]
that we write
\begin{equation}\label{eq:modif}
\DR_{X\times\CC/\CC}(\pi_+\cN')\simeq\bR\pi_*\DR_{X'\times\CC/\CC}(\cN'),
\end{equation}
because $\dim X'=\dim X$. As a consequence of Proposition \ref{prop:Dconst}, if \ref{cor:Dconst}\eqref{cor:Dconst1} holds for a regular holonomic $\cD_{X'\times\CC}$-module $\cN'$, then it holds for $\pi_+\cN'$.
\item
The assertion holds for $\cB_{f'}$ if $f^{\prime-1}(0)$ has normal crossings. This follows from a local computation (which will be done in the proof of Lemma \ref{lem:formel} below). From~\eqref{enum:imdir}, it holds for $\pi_+\cB_{f'}$, if $\pi:X'\to X$ as above is chosen such that $f^{\prime-1}(0)$ has normal crossings. Such a $\pi$ exists if we restrict $X$ to the neighbourhood of any singular point in $X_0$.
\item\label{enum:finimdir}
According to the decomposition theorem of~\cite{MSaito86} (or, \cf\cite[Proof of Prop\ptbl1.1]{B-S07}, a simple variant of it using that $\pi$ can be obtained by successive blowing-ups with non-singular centers), $\cO_X$ is a direct summand of $\pi_+\cO_{X'}$ in $D^b(\cD_X)$. As a consequence, working with $\cM=\cO_X$ and $\cM'=\cO_{X'}$, we find that $\cK_f^\cbbullet$ is a direct summand of $\bR\pi_*\cK_{f'}^\cbbullet$ in $D^b(\cD_{\CC,0})$. We immediately conclude that the assertion holds for $\cB_f$.
\item
When $\cM$ is locally $\cO_X$-free with an integrable connection, we note that the question is local on $X$ and we can assume that $\cL=\CC_X$, hence we can apply the previous arguments.
\end{enumerate}

One can argue similarly for \ref{cor:Dconst}\eqref{cor:Dconst5}. One first checks that it holds for $\cB_{f'}$ (notation as above), \cf \cite[Prop\ptbl2.1]{B-S07}.

One can define a sheaf $V^0(\cD_{\CC\from X\times\CC})$ and define $\pi_+$ for $V^0\cD_{X'\times\CC}$-modules. We get a formula similar to \eqref{eq:modif}. Then \ref{cor:Dconst}\eqref{cor:Dconst5} holds with $j_0=0$ for $\cH^0\pi_+\cB_{f'}$. Indeed, one can apply an argument similar to that of the proof of Corollary \ref{cor:Dconst} (at the level of $\cD_{X\times\CC}$-modules) to deduce that the natural morphism
\[
\cH^i\pi_+(V^j\cB_{f'})\to\cH^i\pi_+(\cB_{f'})
\]
is injective and has image the Kashiwara-Malgrange filtration $V^j\cH^i\pi_+(\cB_{f'})$ (\cf \cite[\S3.3]{MSaito86}, \cite[\S4.8]{M-S86b}, \cite[Prop\ptbl9.2.5]{L-M95}, \cite[Th\ptbl3.1.8]{Bibi01c}). In particular, one uses that $\cH^i\pi_+(V^j\cB_{f'})$ is $V^0\cD_{X\times\CC}$-coherent in order to control its $t$-torsion part (\cf\eg\cite[Lem\ptbl 3.1.4]{Bibi01c}).

We also notice that $\cH^i\pi_+\cB_{f'}$ is supported on $X_0$ if $i\neq0$, hence satisfies $V^j\cH^i\pi_+\cB_{f'}=0$ for $j\leq0$. Setting $\cN=\cH^0\pi_+\cB_{f'}$, we thus have
\begin{align*}
U^0\DR_{X\times\CC/\CC}(\cN)&:=\DR_{X\times\CC/\CC}(V^0\cN)\\
&=\DR_{X\times\CC/\CC}(\cH^0\pi_+V^0\cB_{f'})\\
&\simeq\DR_{X\times\CC/\CC}(\pi_+V^0\cB_{f'})\\
&\simeq\bR\pi_*\DR_{X'\times\CC/\CC}(V^0\cB_{f'})\\
&=:\bR\pi_*U^0\DR_{X'\times\CC/\CC}(\cB_{f'}).
\end{align*}
Since the cohomology sheaves of $U^0\DR_{X'\times\CC/\CC}(\cB_{f'})$ are $\CC\{t\}$-constructible and since~$\pi$ is proper, a spectral sequence argument shows that $U^0\cH^i\DR_{X\times\CC/\CC}(\cN)_x$ has finite type over $\CC\{t\}$ for each $x\in X_0$. Since $\cB_f$ is a direct summand of $\cN$, as already noticed above, we obtain \ref{cor:Dconst}\eqref{cor:Dconst5} for $(\cK_f^\cbbullet,U^\cbbullet\cK_f^\cbbullet)$.
\end{proof}

\subsection{The complex $\wt\cK_f^{\protect\cbbullet}$ and the nearby cycles}
Let us now consider the $\cD_X$-module $\cM[1/f]\simeq\cL\otimes_{\cO_X}\cO_X[1/f]$, where $\cO_X[1/f]$ denotes the sheaf of meromorphic functions on $X$ with poles along $X_0$ at most. Let us set $\wt\cK_f^\cbbullet=\DR_{X\times\CC/\CC}(i_{f,+}\cM[1/f])\simeq\cL\otimes(\Omega_X^{\bbullet}[1/f,\partial_t],d-df\otimes\partial_t)$. This is a complex of $\CC(\{t\})$-vector spaces, equipped with a connection. We have $\cH^i(\wt\cK_f^\cbbullet)=\CC(\{t\})\otimes_{\CC\{t\}}\cH^i(\cK_f^\cbbullet)$, in a way compatible with the connections, hence it is a constructible sheaf of $\CC(\{t\})$-vector spaces, according to Theorem \ref{th:KM}\eqref{th:KM2}. As a consequence, if $X_0$ is the complement of a closed analytic subset of a compact analytic space $\ov X_0$ (for example, if $X_0$ is an algebraic variety, as in \S\ref{sec:intro}, by using a Nagata compactification), then the hypercohomology $\bH^i(X_0,\wt\cK_f^\cbbullet)$ is a finite dimensional $\CC(\{t\})$-vector space, equipped with a connection $\nabla_{\partial_t}$. On the other hand, with the same assumption on $X_0$, the complex $\psi_f\cL$ equipped with its monodromy $\rT_f$ gives rise to finite dimensional $\CC$-vector spaces $\bH^i(X_0,\psi_f\cL)$, equipped with an automorphism $\rT_f$, and then to a finite dimensional $\CC(\{t\})$-vector space with connection $\wtRHm\big(\bH^i(X_0,\psi_f\cL),\rT_f\big)$.

\begin{corollaire}\label{cor:rhpsi}
Under the previous assumption on $X_0$, we have for each $i$:
\[
\big(\bH^{i+1}(X_0,\wt\cK_f^\cbbullet),\nabla_{\partial_t}\big)\simeq\wtRHm\big(\bH^i(X_0,\psi_f\cL),\rT_f\big).
\]
\end{corollaire}

Notice that the left-hand term above is nothing but
\[
\Big(\bH^{i+1}\big(X_0,(\Omega_X^\cbbullet[1/f,\partial_t],d-df\otimes\partial_t\big),\nabla_{\partial_t}\Big).
\]

\begin{proof}
According to Theorem \ref{th:KM}\eqref{th:KM2}, we can apply \eqref{eq:Dconsttilde**} to $(\cK_f^\cbbullet,U^\cbbullet\cK_f^\cbbullet)$. The result follows from the identification of $(\gr^0_U\cK_f^\cbbullet,\exp(-2\pi it\partial_t))$ with $(\psi_f\cL[-1],\rT_f)$ given by Theorem \ref{th:KM}\eqref{th:KM1}.
\end{proof}

\subsection{The complex $\wh\cK_f^{\protect\cbbullet}$ and the vanishing cycles}
We denote by $_X\wh\cK_f^\cbbullet$ the complex $\DR_X(\wh\cE_X^{-f/u}\otimes\cM)\simeq\cL\otimes_\CC(\Omega_X^\cbbullet\lpr\hb\rpr,d-df/\hb)$. Our goal is to prove the microlocal version of Corollary \ref{cor:rhpsi}. Recall that $\wh\cK_f^\cbbullet$ denotes $\wh\cE_{\CC,0}\otimes_{\cD_{\CC,0}}^{\bL}\cK_f^\cbbullet\simeq\CC\lpr\partial_t^{-1}\rpr\otimes_{\CC[\partial_t]}\DR_{X\times\CC/\CC}(\cM_f)$, if we identify $\partial_t^{-1}$ with $\hb$.

\begin{proposition}\label{prop:BSK}
Under the previous assumption on $X_0$, we have:
\[
\Big(\bH^{i+1}\big(X_0,\cL\otimes(\Omega^\cbbullet_X\lpr\hb\rpr,d-df/u)\big),\nabla_{\partial_u}\Big)\simeq\whRHm\big(\bH^i(X_0,\phi_f\cL),\rT_f\big).
\]
\end{proposition}

\begin{proof}
By definition, the left-hand term is identified with $\bH^{i+1}(X_0,{}_X\wh\cK_f^\cbbullet)$, and an argument similar to that of Corollary \ref{cor:rhpsi}, by using Corollary \ref{cor:Dconsthat} instead of \ref{cor:Dconsttilde}, identifies the right-hand term with $\bH^{i+1}(X_0,\wh\cK_f^\cbbullet)$. The point is then to compare $_X\wh\cK_f^\cbbullet$ and $\wh\cK_f^\cbbullet$. We have a natural morphism
\begin{equation}\label{eq:morphismemicros}
\wh\cK_f^\cbbullet:=\wh\cE_{\CC,0} \otimes_{\cD_{\CC,0}}\cK^\cbbullet_f\to{}_X\wh\cK_f^\cbbullet,
\end{equation}
by sending $\big(\sum_{k\leq k_o}\partial_t^kb_k(t)\big)\otimes\omega$ to $\big(\sum_{k\leq k_o}b_k(f)\hb^{-k}\big)\otimes\omega$. The proposition is now a consequence of Lemma \ref{lem:formel} below.
\end{proof}

\begin{remarque}
The cohomology sheaves of $_X\wh\cK_f^\cbbullet$ are supported on the critical locus of~$f$. Indeed, assume that $f=x_1$ in some local system of coordinates of~$X$. The complex $_X\wh\cK_f^\cbbullet$ is the simple complex associated to the cube with vertices $\cO_X\lpr\hb\rpr$ and arrows $\partial_{x_i}-\partial_{x_i}(f)\hbm$. It is thus locally isomorphic to the complex
\[
\CC\{x_1\}\lpr\hb\rpr\To{u\partial_{x_1}-1}\CC\{x_1\}\lpr\hb\rpr,
\]
hence clearly quasi-isomorphic to zero.
\end{remarque}

\begin{lemme}\label{lem:formel}
The natural morphism \eqref{eq:morphismemicros} is a quasi-isomorphism.
\end{lemme}

\begin{remarque}
It follows from this lemma and from Remark \ref{rem:whF} that $_X\wh\cK_f^\cbbullet$ is a constructible sheaf of $\CC\lpr\hb\rpr$-vector spaces on $X_0$. This was already obtained in \cite[Prop\ptbl3.9]{Kapranov91} by using the statement of Lemma \ref{lem:formel}, without proof however.

Notice also that Lemma \ref{lem:formel} together with \eqref{eq:Dconsthat*} (according to Theorem \ref{th:KM}\eqref{th:KM2}) gives $\cH^{i+1}(\wh\cK_f)\simeq\whRHm(\cH^i\phi_f\cL,\rT_f)$, which was obtained in \cite[Th\ptbl2.4(a)]{Kapranov91} only after a suitable grading.
\end{remarque}

\begin{proof}[Proof of Lemma \ref{lem:formel}: the normal crossing case]
Let us start with the case where $f$ is a monomial, in local coordinates, so that $X_0$ is a divisor with normal crossings.

The corresponding logarithmic complexes $\cK_f(\log X_0)^\cbbullet$ and $_X\wh{\cK_f(\log X_0)}{}^\cbbullet$ are naturally filtered by the weight filtration (number of polar divisors) $W_\ell$, with $W_0=\cK_f^\cbbullet$ or $_X\wh\cK_f^\cbbullet$. It is therefore enough to show the desired statement for the logarithmic complexes and for each $\gr_\ell^W$ with $\ell\geq1$. Let $f=x^\mu$, with $\mu_1,\dots,\mu_r\geq1$, $1\leq r\leq n$. We will prove the quasi-isomorphism at the origin of coordinates.

Let us start with $\gr_\ell^W$, with $\ell\geq1$. Then $df\otimes\partial_t$ or $df\hbm$ induces zero on $\gr_\ell^W$, so we find, when setting $X_{0,I}=\bigcap_{i\in I}\{x_i=0\}$,
\begin{align*}
\gr_\ell^W(\cK_f(\log X_0)^\cbbullet)&=\bigoplus_{|I|=\ell}\CC[\partial_t]\otimes_\CC\DR(\cO_{X_{0,I}}),\\
\gr_\ell^W({}_X\wh{\cK_f(\log X_0)}{}^\cbbullet)&=\bigoplus_{|I|=\ell}\DR(\cO_{X_{0,I}}\lpr\hb\rpr),
\end{align*}
and it is easy to check that \eqref{eq:morphismemicros} is an isomorphism in this case.

It is now enough to prove the isomorphism for the germs at each point $x_o$ of $\{x_1=\cdots=x_r=0\}$ of the logarithmic complexes. We will assume that $x_{r+1}(x_o)=\cdots=x_n(x_o)=0$ for the sake of simplicity. We will identify the quotient module $\cO_{X,0}/x^\mu$ with the module
\[
M=\bigoplus_{\emptyset\neq I\subset\{1,\dots,r\}}\bigoplus_{\ell_I}M_{I,\ell_I},\quad M_{I,\ell_I}\defin x^{\ell_I}x^{\mu_{I^c}}\cO_{X_I,0},
\]
where $X_I=\{x_i=0\mid\forall i\in I\}$, $x^{\mu_{I^c}}=\prod_{i\notin I}x_i^{\mu_i}$ and, for each $i\in I$, $\ell_i$ varies in $\{0,\dots,\mu_i-1\}$. Let us set $d=\gcd(\mu_1,\dots,\mu_r)$ and $\mu'=\mu/d$. We will use that the simple complex associated with the $(n-1)$-cube having vertices~$M$ and differentials $x_i\partial_{x_i}-(\mu_i/\mu_1)x_1\partial_{x_1}$ if $i=2,\dots,r$, and $\partial_{x_i}$ if $i\geq r+1$ (\ie the Koszul complex on $M$ with these differentials) is quasi-isomorphic to the subcomplex having vertices $\bigoplus_{j=0}^{d-1}x^{j\mu'}\CC$ and differentials equal to zero. Indeed, if $I\neq\{1,\dots,r\}$, let us choose $i\in\{2,\dots,r\}$ such that one and only one between~$1$ and~$i$ belongs to~$I$. Then $x_i\partial_{x_i}-(\mu_i/\mu_1)x_1\partial_{x_1}$ is bijective on $M_{I,\ell_I}$ for each~$\ell_I$, since it acts either as $x_i\partial_{x_i}+\mu_i(1-\ell_1/\mu_1)$ or as $-(\mu_i/\mu_1)\big(x_1\partial_{x_1}+\nobreak\mu_1(1-\nobreak\ell_i/\mu_i)\big)$ on $\cO_{X_I}$. It remains to prove the assertion for the Koszul complex when $I=\{1,\dots,r\}$, for which it is clear.

The complex $\cK_f(\log X_0)^\cbbullet_0$ is the simple complex associated with the $n$-cube of vertices $\cO_{X,0}[\partial_t]$ and arrows $x_i\partial_{x_i}-\mu_ix^\mu\partial_t$. We notice that the map $x_1\partial_{x_1}-\mu_1x^\mu\partial_t$ is injective, and its cokernel can be identified with $N\defin\cO_{X,0}\oplus\partial_t(\cO_{X,0}/x^\mu)[\partial_t]$. On this cokernel, the action of $x_i\partial_{x_i}-\mu_ix^\mu\partial_t$ ($i\in\{2,\dots,r\}$) is identified with the action induced by $x_i\partial_{x_i}-(\mu_i/\mu_1)x_1\partial_{x_1}$. Therefore, $\cK_f(\log X_0)^\cbbullet_0$ is isomorphic to the simple complex associated with the $(n-1)$-cube with vertices $N$ and arrows $x_i\partial_{x_i}-(\mu_i/\mu_1)x_1\partial_{x_1}$ ($i\in\{2,\dots,r\}$) and $\partial_{x_i}$ ($i\geq r+1$). From the preliminary remark above we deduce that the inclusion of the sub-cube of size $r-1$ having vertices
\[
\CC\{x^{\mu'}\}\oplus\bigoplus_{j=0}^{d-1}\partial_t\CC[\partial_t]\cdot x^{j\mu'}=\bigoplus_{j=0}^{d-1}\Big(\CC\{x^\mu\}\oplus\partial_t\CC[\partial_t]\Big)\cdot x^{j\mu'},
\]
and induced arrows equal to zero, is a quasi-isomorphism. Let us set $e_j=x^{j\mu'}$ for $j=0,\dots,d-1$. The induced action of $\partial_t$ sends $x^{k\mu}e_j$ with $k\geq 1$ to $(k+j/d)x^{(k-1)\mu}$, and~$e_j$ to $\partial_t\cdot e_j$. For $g\in\CC\{t\}$, we also have $g(t)e_j=g(x^\mu)e_j$, and the value of $g(t)\partial_t^\ell e_j$ is computed with the previous ones, according to the standard commutation relations. One checks that $e_j$ satisfies $(t\partial_t+1-j/d)e_j=0$, and that $\big(\CC\{x^\mu\}\oplus\partial_t\CC[\partial_t]\big)\cdot e_j\simeq\cD_{\CC,0}/\cD_{\CC,0}(t\partial_t+1-j/d)$. Since $\wh\cE_{\CC,0}$ is flat over $\cD_{\CC,0}$, the same reasoning applies to $\wh{\cK_f^{\cbbullet}(\log X_0)}_0$, which is thus quasi-isomorphic to the subcomplex having vertices $\bigoplus_{j=0}^{d-1}[\wh\cE_{\CC,0}/\wh\cE_{\CC,0}(t\partial_t+1-j/d)]e_j$ and arrows equal to zero. We finally notice that each term in the sum is written as $\CC\lpr\partial_t^{-1}\rpr\cdot e_j$, so that $\wh{\cK_f^{\cbbullet}(\log X_0)}_0$ is quasi-isomorphic to the subcomplex having vertices $\bigoplus_{j=0}^{d-1}\CC\lpr\partial_t^{-1}\rpr\cdot e_j$ and arrows equal to zero.

A similar computation can be done for $_X\wh{\cK_f(\log X_0)}{}^\cbbullet$. It is isomorphic to the simple complex associated to the $n$-cube with vertices $\cO_X\lpr\hb\rpr$ and arrows as above. We can replace each arrow indexed by $i\in\{2,\dots,r\}$ with the linear combination $x_i\partial_{x_i}-(\mu_i/\mu_1)x_1\partial_{x_1}$, and get a complex isomorphic to the original one. Moreover, we can also replace $x_1\partial_{x_1}-\mu_1x^\mu\hbm$ with $\hb x_1\partial_{x_1}-\mu_1x^\mu$. In such a way, our complex is obtained by $\CC\lpr\hb\rpr\otimes_{\CC\lcr\hb\rcr}$ from the $n$-cube with vertices $\cO_X\lcr\hb\rcr$ and similar arrows. We will work with this complex, that we denote by $K^\cbbullet$, and tensor with $\CC\lpr\hb\rpr$ at the very end. The germ $K^\cbbullet_0$ at the origin is equal to $\varinjlim_{U\supset 0}\Gamma(U,K^\cbbullet)$, where $U$ varies in a fundamental system of polycylinders neighbouring the origin.

For any such polycylinder $U$, each arrow $\hb x_1\partial_{x_1}-\mu_1x^\mu$ is injective on the vertex $\cO(U)\lcr\hb\rcr$, and the cokernel is identified with $\wh M\defin(\cO(U)/x^\mu)\lcr\hb\rcr$. The $(n-1)$-cube with vertices all equal to $\wh M$ and arrows $x_i\partial_{x_i}-(\mu_i/\mu_1)x_1\partial_{x_1}$ ($i=2,\dots,r$) and $\partial_{x_i}$ ($i\geq r+1$) is quasi-isomorphic to the sub-cube of size $r-1$ having vertices
\[
\bigoplus_{j=0}^{d-1}\CC\lcr\hb\rcr\cdot x^{j\mu'},
\]
and induced arrows equal to zero. Therefore, $K^\cbbullet_0$ is isomorphic to the same complex. Lastly, $_X\wh{\cK_f(\log X_0)}{}^\cbbullet_0$ is isomorphic to $\CC\lpr\hb\rpr\otimes_{\CC\lcr\hb\rcr}K^\cbbullet_0$, hence to the subcomplex having vertices $\bigoplus_{j=0}^{d-1}\CC\lpr\hb\rpr\cdot x^{j\mu'}$ and arrows equal to zero. That \eqref{eq:morphismemicros} is a quasi-isomorphism is now clear.
\end{proof}

\begin{proof}[Proof of Lemma \ref{lem:formel}: the general case]
Since the statement is local, we can work in the neighbourhood of a point $x_o\in X$, and we can find an embedded resolution $\pi:X'\to X$ of $f^{-1}(0)$ in the neighbourhood of this point, that we still denote by $X$. We can also assume that $\cL=\CC_X$ in this local setting. The statement of Lemma \ref{lem:formel} holds for $\wh\cK_{f'}$ and ${}_{X'}\wh\cK_{f'}$, with $f'=f\circ\pi$, according to the previous computation.

Let us apply $\bR\pi_*$ to the isomorphism \eqref{eq:morphismemicros} for $\cK^\cbbullet_{f'}$. Since $\pi$ is proper, we deduce from the projection formula together with \eqref{eq:modif} that the left-hand term that we get is the left-hand term of \eqref{eq:morphismemicros} for $\pi_+\cO_{X'}$. According to Proposition \ref{prop:imdiru} (with $D=\emptyset$), the right-hand term is isomorphic to $\DR(\wh\cE_X^{-f/u}\otimes\pi_+\cO_{X'})$ (recall that $\dim X=\dim X'$).

Since $\cO_X$ is a direct summand of $\pi_+\cO_{X'}$ by the decomposition theorem of \cite{MSaito86}, it follows that \eqref{eq:morphismemicros} for $\cK^\cbbullet_f$ is an isomorphism.
\end{proof}

\section{Regularity}
We keep the setting of \S\ref{subsec:gaga}. As in \S\ref{sec:andR}, we will work in the analytic topology, and we will not use the exponent $\an$. We denote by $\cO_Y(*D)$ the sheaf of meromorphic functions on $Y$ with pole along~$D$ at most (a divisor that we do not assume with normal crossings at the moment), and we denote by $j:X=Y\moins D\hto Y$ the open inclusion. Let $\cM$ be a coherent $\cO_Y(*D)$-module of rank $d$ equipped with a flat connection $\nabla:\cM\to\Omega_Y^1\otimes_{\cO_Y}\cM$ (we know by \cite{Malgrange95} that~$\cM$ is then locally stably free). Our goal in this section is to prove:

\begin{proposition}\label{prop:comparaison}
Assume that $\nabla$ has regular singularity along $D$. Then, the natural morphism of complexes
\bgroup\numstareq
\begin{equation}\label{eq:comparaison}
\DR(\wh\cE_{(Y,D)}^{-F/u}\otimes\cM)\to \bR j_*\DR(\wh\cE_X^{-f/u}\otimes\cM_{|X})
\end{equation}
\egroup
is a quasi-isomorphism.
\end{proposition}

\begin{remarques}\label{rem:stein}\mbox{}
\begin{enumerate}
\item\label{rem:stein1b}
Since $\cM$ is $\cO_Y(*D)$-coherent, we can use  \eqref{eq:compu} to express \eqref{eq:comparaison}, which is equivalent to the isomorphism
\[
(\cM\lpr\hb\rpr,\nabla-\hbm dF)\isom\bR j_*(\cM_{|X}\lpr\hb\rpr,\nabla-\hbm df).
\]

\item\label{rem:stein2}
In an exact sequence $0\to\cM'\to\cM\to\cM''\to0$, if the result holds for $\cM'$ and $\cM''$, it holds for $\cM$. Therefore, we can assume that $\cM$ is a \emph{simple} meromorphic bundle with connection.
\item\label{rem:stein3}
Assume we have a commutative diagram of data as above with $\pi$ proper:
\[
\xymatrix@R=.2cm{
&Y'\ar[dd]^\pi\\X\ar@{_{ (}->}[dr]_-j\ar@{^{ (}->}[ur]^(.6){j'}&\\&Y
}
\]
and let us set $F'=F\circ\pi$. It follows from Proposition \ref{prop:imdiru} that, if the proposition holds for $(Y',j',F')$, it holds for $(Y,j,F)$ by applying $\bR\pi_*$ to both members of \eqref{eq:comparaison}. We can therefore assume that $D$ is a divisor with normal crossings. In such a case, $\cM$ is $\cO_Y(*D)$-locally free and we denote by $d$ its rank.
\end{enumerate}
\end{remarques}

\subsection{Along $F=\infty$}\label{subsec:F=infty}
Let us set $D_\infty=F^{-1}(\infty)$. Near a point $y_o$ of $D_\infty$, we can find a system of local coordinates $(x,y,z)$ of $Y$ such that $F(x,y,z)=y^{-\nu}=\prod_{j=1}^ny_j^{-\nu_j}$, with $\nu_j\geq1$ for $j=1,\dots,n$, and $D=\{y_1\cdots y_n\cdot z_1\cdots z_p=0\}$. If we still denote by $Y$ a small open neighbourhood of $y_o$, we can assume that $(\cM,\nabla)=(\cO_Y(*D)^d,d+A)$, where~$A$ is a matrix of $1$-forms which are logarithmic along $D$. We can moreover assume that $A$ is written as
\begin{equation}\label{eq:matrices}
A=\sum_jA_{y_j}\frac{dy_j}{y_j}+\sum_kA_{z_k}\frac{dz_k}{z_k},
\end{equation}
where $A_{y_j},A_{z_k}$ are pairwise commuting constant matrices whose eigenvalues have their real part in $[0,1)$. We realize the complex $\DR(\wh\cE_{(Y,D)}^{-F/u}\otimes\cM)$ (\resp $\DR(\wh\cE_X^{-f/u}\otimes\nobreak\cM_{|X})$) as the simple complex associated to the $\dim Y$-dimensional cube having vertices equal to $\cO_Y(*D)^d\lpr\hb\rpr$ (\resp $\cO_X^d\lpr\hb\rpr$) and arrows equal to
\begin{align*}
\partial_{x_i}\ (i=1,\dots,m),\
\partial_{y_j}+\frac{A_{y_j}}{y_j}-\nu_j\frac{y^{-\nu-1_j}}\hb\ (j=1,\dots,n),\
\partial_{z_k}+\frac{A_{z_k}}{z_k}\ (k=1,\dots,p).
\end{align*}
One easily checks that $y_j\partial_{y_j}+A_{y_j}+\nu_jy^{-\nu}/\hb$ induces an isomorphism from $\cO_Y(*D)^d\lpr\hb\rpr$ to itself and from $\cO_X^d\lpr\hb\rpr$ to itself. Therefore, both complexes in \eqref{eq:comparaison} are zero near~$y_o$, and Proposition \ref{prop:comparaison} holds at this point.

\subsection{Along $F=c$, first reduction}
We now fix $y_o\in D$ such that $F(y_o)=c\neq\infty$. Since it is harmless for the statement to replace $F$ with $F-c$, we can assume that $c=0$. There exists a projective modification $\varpi:Y'\to Y$ with $Y'$ smooth such that, setting $D'=\varpi^{-1}(D)$ and $F'=F\circ\varpi$, $F^{\prime-1}(0)\cup D'$ is a divisor with normal crossings in $Y'$. Then the locally free $\cO_{Y'}(*D')$-module $\varpi^*\cM$ is naturally equipped with a meromorphic connection with poles along $D'$ and regular singularities. We denote by $\varpi^+\cM$ this bundle with flat connection.

Let us consider the minimal extension (as $\cD$-modules) $j_{!*}\cM$ and $j'_{!*}\varpi^+\cM$. Since~$\cM$ is assumed to be simple (\cf Remark \ref{rem:stein}\eqref{rem:stein2}), $\varpi^+\cM$ is also simple (because simplicity is preserved by restriction to Zariski open sets, and $\cM$ and $\varpi^+\cM$ coincide on $Y\moins(D\cup F^{-1}(0))$). Then $j_{!*}\cM$ and $j'_{!*}\varpi^+\cM$ are simple as $\cD$-modules. From the decomposition theorem for simple $\cD$-modules (\cf \cite{Mochizuki07}), we conclude that $j_{!*}\cM$ is a direct summand of $\varpi_+j'_{!*}\varpi^+\cM$, where $\varpi_+$ denotes the direct image of $\cD$-modules. Tensoring with $\cO_Y(*D)$, we find that $\cM$ is a direct summand of $\varpi_+\varpi^+\cM$, where now~$\varpi_+$ denotes the direct image from the category of $\cD_{Y'}(*D')$-modules to the derived category of $\cD_Y(*D)$-modules. We can now argue as in the last part of the proof of Lemma \ref{lem:formel}, by using Proposition \ref{prop:imdiru}, to conclude that the morphism $\bR\varpi_*\eqref{eq:comparaison}(\varpi^+\cM)$ has the morphism $\eqref{eq:comparaison}(\cM)$ as direct summand, and therefore the latter is an isomorphism as soon as the former is so. This reduces the proof of Proposition \ref{prop:comparaison} to the case where $F^{-1}(0)\cup D$ has normal crossings and $y_o\in F^{-1}(0)\cap D$, which is considered in the next subsection.

\subsection{Along $F=c$, the normal crossing case}
We consider the following local setting: the space $Y$ is a polydisc $\Delta^m\times\Delta^n\times\Delta^p\times\Delta^q$ with set of multi-coordinates $x,y,z,t$. The divisor $D$ is defined as $\bigcup_j\{y_j=0\}\cup\bigcup_k\{z_k=0\}$ and $F$ is the monomial $x^\mu y^\nu=x_1^{\mu_1}\cdots x_m^{\mu_m}\cdot y_1^{\nu_1}\cdots y_n^{\nu_n}$, with $\mu_i,\nu_j\geq1$. The variables $t$ are parameters, and we will soon neglect them.

We assume that $(\cM,\nabla)$ is as in \S\ref{subsec:F=infty}. We can assume that the matrices $A_{y_j},A_{z_k}$ of \eqref{eq:matrices} are in the Jordan normal form, so that $(\cM,\nabla)$ is an extension of rank-one meromorphic bundles with flat connection. It is therefore enough to prove the proposition when $\cM$ has rank one. We will then set $d=1$ and denote by $a_{y_j},a_{z_k}$ the corresponding ``matrices''.

We realize the complex $\DR(\wh\cE_{(Y,D)}^{-F/u}\otimes\cM)$ (\resp $\DR(\wh\cE_X^{-f/u}\otimes\cM_{|X})$) as the simple complex associated to the $\dim Y$-dimensional cube having vertices equal to $\cO_Y(*D)\lpr\hb\rpr$ (\resp $\cO_X\lpr\hb\rpr$) and arrows equal to
\begin{align*}
\nabla^f_{x_i}&\defin\partial_{x_i}+\mu_i\frac{x^{\mu-1_i}y^\nu}\hb\quad\text{for $i=1,\dots,m$},\\
\nabla^f_{y_j}&\defin\partial_{y_j}+\frac{a_{y_j}}{y_j}+\nu_j\frac{x^\mu y^{\nu-1_j}}\hb\quad\text{for $j=1,\dots,n$},\\
\nabla^f_{z_k}&\defin\partial_{z_k}+\frac{a_{z_k}}{z_k}\quad\text{for $k=1,\dots,p$},\\
\nabla^f_{t_\ell}&\defin\partial_{t_\ell}\quad\text{for $\ell=1,\dots,q$}.
\end{align*}
By integrability, the components of the operator $\nabla^f$ pairwise commute. We can easily reduce to the kernel of the components $\nabla^f_{t_\ell}$, and forget them.

\subsubsection*{First case: $m=0$}
In this case, there is no variable $x$. If $n=0$, the result of Proposition \ref{prop:comparaison} is clear, as $dF=0$ and by standard results for regular connections (\cf Lemma \ref{lem:comparaison} below). We will therefore assume that $n\geq1$. In such a case, $\hb\nabla^f_{y_1}$ is bijective both on $\cO_Y(*D)\lpr\hb\rpr$ and on $\cO_X\lpr\hb\rpr$, and therefore both complexes in \eqref{eq:comparaison} are quasi-isomorphic to $0$. Indeed, the coefficient $\psi_k$ of $\hb^k$ in $\psi=\hb\nabla^f_{y_1}(\sum_{\ell\geq\ell_o}\varphi_\ell\hb^\ell)$ is $(\partial_{y_1}+a_1/y_1)\varphi_{k-1}+\nu_1y^{\nu-1_1}\varphi_k$. Assume $\varphi_\ell=0$ for $\ell<\ell_o$ and $\varphi_{\ell_o}\neq0$. Then $\psi_\ell=0$ for $\ell<\ell_o$ and $\psi_{\ell_o}=\nu_1y^{\nu-1_1}\varphi_{\ell_o}$. This implies injectivity, since $\nu_1\neq0$. We get surjectivity by using that $y_j$ are invertible in $\cO_Y(*D)$ or $\cO_X$.

\subsubsection*{Second case: $m>0$}
Let us denote by $D_x$ the divisor defined by $x_1\cdots x_m=0$. We first replace the de~Rham complexes of Proposition \ref{prop:comparaison} with the \emph{logarithmic} de~Rham complexes with respect to $D_x$ (and still meromorphic with respect to $D$ for the left-hand complex).

\begin{proposition}\label{prop:logcomparaison}
With such a setting, Proposition \ref{prop:comparaison} holds with $\Omega_Y^\cbbullet(\log D_x)$ instead of~$\Omega_Y^\cbbullet$, and similarly for $\Omega_X^\cbbullet$.
\end{proposition}

\begin{proof}
We can now work with the Koszul complexes associated to the $\dim Y$-dimensional cube having vertices equal to $\cO_Y(*D)\lpr\hb\rpr$ (\resp $\cO_X\lpr\hb\rpr$) and differentials $x_i\nabla^f_{x_i}$, $\nabla^f_{y_j},\nabla^f_{z_k}$ (or equivalently $y_j\nabla^f_{y_j},z_k\nabla^f_{z_k}$), and $\nabla^f_{t_\ell}$. Moreover, by changing the basis of the vector space underlying the Koszul complex, it is equivalent to consider the differentials $x_1\nabla^f_{x_1}$, $x_i\partial_{x_i}-(\mu_i/\mu_1)x_1\partial_{x_1}$ ($i=2,\dots,m$), $y_j\partial_{y_j}+a_{y_j}-(\nu_j/\mu_1)x_1\partial_{x_1}$ ($j=1,\dots,n$), $z_k\partial_{z_k}+a_k$ ($k=1,\dots,p$), $\partial_{t_\ell}$ ($\ell=1,\dots,q$). Lastly, since $\hb$ is invertible on the terms of this complex, we will work with $\hb x_1\nabla^f_{x_1}$.

We denote by $K^\cbbullet_Y(*D)\lcr\hb\rcr$ (\resp $K^\cbbullet_X\lcr\hb\rcr$ the complex associated to the cube with vertices $\cO_Y(*D)\lcr\hb\rcr$ (\resp $\cO_X\lcr\hb\rcr$) and arrows given by the previous formulas.

\begin{lemme}\label{lem:comparisonlcr}
The natural morphism $K^\cbbullet_Y(*D)\lcr\hb\rcr\to\bR j_*K^\cbbullet_X\lcr\hb\rcr$ is an isomorphism.
\end{lemme}

\begin{proof}
Let us consider the germ of this morphism at the origin. Let $U_\epsilon$ be the family of polydiscs of poly-radius $\epsilon$ in the previous coordinate system. We notice that $R^k\Gamma(U_\epsilon\moins\nobreak D,K^\ell_X\lcr\hb\rcr)=0$ for each $\ell$ and each $k>0$, since $U_\epsilon\moins D$ is Stein, according to the remarks in \S\ref{subsec:prelim}. Therefore,
\begin{multline}\label{eq:limU}
\big[K^\cbbullet_Y(*D)\lcr\hb\rcr_0\ra\bR j_*K^\cbbullet_X\lcr\hb\rcr_0\big]\\
=\varinjlim_{\epsilon>0}\big[\Gamma\big(U_\epsilon,K^\cbbullet_Y(*D)\lcr\hb\rcr\big)\ra\Gamma\big(U_\epsilon\moins D,K^\cbbullet_X\lcr\hb\rcr]\big)\big].
\end{multline}
We now set $U=U_\epsilon$. Then $\Gamma\big(U,K^\cbbullet_Y(*D)\lcr\hb\rcr\big)$ (\resp $\Gamma\big(U\moins\nobreak D,K^\cbbullet_X\lcr\hb\rcr]\big)$) is the simple complex associated to the cube with vertices $\cO(U)(*D)\lcr\hb\rcr$ (\resp $\cO(U\moins\nobreak D)\lcr\hb\rcr$) and arrows $\hb x_1\nabla^f_{x_1}$, $x_i\partial_{x_i}-(\mu_i/\mu_1)x_1\partial_{x_1}$ ($i=2,\dots,m$), $y_j\partial_{y_j}+a_{y_j}-(\nu_j/\mu_1)x_1\partial_{x_1}$ ($j=1,\dots,n$), $z_k\partial_{z_k}+a_k$ ($k=1,\dots,p$), $\partial_{t_\ell}$ ($\ell=1,\dots,q$). Moreover, it is not difficult to reduce to the kernel of the $\partial_{t_\ell}$, so we will simply forget these coordinates in order to simplify the notation.

We will use the following notation. For $I\subset\{1,2,\dots,m\}$, $U_{\wh I}$ denotes the product of the discs $\Delta_\epsilon$ except the discs with coordinates $x_i$, $i\in I$. For such a subset $I$, we denote by $[1,\mu_I]$ the product $\prod_{i\in I}\{1,\dots,\mu_i\}$. We set $x^{\prime\mu'}=x_2^{\mu_2}\cdots x_m^{\mu_m}$. We decompose $\cO(U)(*D)$ in the following way:
\begin{multline}\label{eq:decomp}
\cO(U)(*D)\lcr\hb\rcr=\cO(U)(*D)\lcr\hb\rcr x^\mu\\
\oplus\Big(\bigoplus_{\ell_1=0}^{\mu_1-1}\cO(U_{\wh 1})(*D)\lcr\hb\rcr x_1^{\ell_1}\Big)\oplus\Big(\bigoplus_{\emptyset\neq I\subset\{2,\dots,m\}}\bigoplus_{\ell_I\in[1,\mu_I]}\cO(U_{\wh I})(*D)\lcr\hb\rcr x^{\mu-\ell_I}\Big),
\end{multline}
and we have a similar decomposition for $\cO(U\moins D)\lcr\hb\rcr$. Let us check that $\hb x_1\nabla^f_{x_1}$ is injective: the coefficient $\psi_k\in\cO(U)(*D)$ of $\hb^k$ in $\psi=\hb x_1\nabla^f_{x_1}(\sum_{\ell\geq\ell_o}\varphi_\ell\hb^\ell)$ is $x_1\partial_{x_1}\varphi_{k-1}+\mu_1x^\mu y^\nu\varphi_k$, hence the assertion is clear (and similarly for $\cO(U\moins\nobreak D)\lcr\hb\rcr$). One can then replace each complex (up to a shift by one) by the corresponding complex for which the terms are the $\coker\Gamma(U,\hb x_1\nabla^f_{x_1})$ and the arrows are those induced by the previous ones. Notice that they are now independent of $\hb$. In the following, we will use the notation $\hb x_1\nabla^f_{x_1}$ instead of  $\Gamma(U,\hb x_1\nabla^f_{x_1})$.

On the other hand, let $\psi\in\cO(U)(*D)\lcr\hb\rcr$ be written as $\sum_{\ell\geq\ell_o}\psi_\ell\hb^\ell$. If $\psi$ is in the image of $\hb x_1\nabla^f_{x_1}$, then $\psi_{\ell_o}\in x^\mu\cO(U)(*D)$ (\resp $x^\mu \Gamma(U\moins D,\cO_X)$), and arguing iteratively with respect to $\ell$, one finds that any $\eta\in\cO(U)(*D)\lcr\hb\rcr)$ (\resp...) is equivalent, modulo the image of $\hb x_1\nabla^f_{x_1}$ to a unique element in the second line of \eqref{eq:decomp} (\resp...). We also note that this identification of $\coker\hb x_1\nabla^f_{x_1}$ with the second line of \eqref{eq:decomp} is $\cO_{U_{\wh 1}}(*D)\lcr\hb\rcr$-linear.

We will describe the differential structure of $\coker\hb x_1\nabla^f_{x_1}$ when using this representation. We will denote by $x_i\wt\nabla^f_{x_i}$, etc., the differentials induced on this cokernel complex. Let us start with the first term. The action of $x_i\wt\nabla^f_{x_i}$ ($i\geq2$) on $\cO(U_{\wh 1})(*D)\lcr\hb\rcr x_1^{\ell_1}$ corresponds to the action of $x_i\partial_{x_i}-\ell_1\mu_i/\mu_1\id$ on $\cO(U_{\wh 1})(*D)\lcr\hb\rcr$, that of $y_j\wt\nabla^f_{y_j}$ to that of $y_j\partial_{y_j}+a_{y_j}-\ell_1\nu_j/\mu_1\id$, and that of $z_k\wt\nabla^f_{z_k}$ to that of $z_k\partial_{z_k}+a_{z_k}$. Therefore, $\wt\nabla^f$ induces on it a $(D_{x'}\cup D)$-logarithmic connection, and we can apply Lemma \ref{lem:comparaison} below to it (replacing $U$ with $U_{\wh1}$ and working on each summand) to conclude the corresponding logarithmic comparison result for it.

We now consider the second term in the second line of \eqref{eq:decomp}. The terms of this direct sum are not themselves $\cO(U_{\wh1})(*D)$-submodules. However, the increasing filtration $F_p$ indexed by $p=\#I$ consists of $\cO(U_{\wh1})(*D)$-submodules, which are stable by the action of the induced operators $x_i\wt\nabla^f_{x_i}$ ($i\geq2$), $y_j\wt\nabla^f_{y_j}$ and $z_k\wt\nabla^f_{z_k}$. Indeed, the differentiable structure is given by the following formulas ($i\geq2$, $\emptyset\neq I\subset\{2,\dots,m\}$):
\begin{align*}
x_i\wt\nabla^f_{x_i}\Big(g(x_{\wh I},y,z)x^{\mu-\ell_I}\Big)&=
\begin{cases}
-\Big(\ell_ig+\dfrac{\mu_i}{\mu_1}x_1\dfrac{\partial g}{\partial x_1}\Big)x^{\mu-\ell_I}&\text{if $i\in I$},\\
\Big(x_i\partial_{x_i}(g)-\dfrac{\mu_i}{\mu_1}x_1\dfrac{\partial g}{\partial x_1}\Big)x^{\mu-\ell_I}&\text{if $i\notin I$},
\end{cases}\\
y_j\wt\nabla^f_{y_j}\Big(g(x_{\wh I},y,z)x^{\mu-\ell_I}\Big)&=
\Big(y_j\partial_{y_j}(g)+a_{y_j}g-\dfrac{\nu_j}{\mu_1}x_1\dfrac{\partial g}{\partial x_1}\Big)x^{\mu-\ell_I},\\
z_k\wt\nabla^f_{z_k}\Big(g(x_{\wh I},y,z)x^{\mu-\ell_I}\Big)&=
\Big(z_k\partial_{z_k}(g)+a_{z_k}g\Big)x^{\mu-\ell_I}.
\end{align*}

As a $\cO(U_{\wh1})(*D)$-module, each graded module $\gr_p^F$ is the direct sum over the subsets~$I$ with $\#I=p$ of the $\cO(U_{\wh{1,I}})(*D)$-modules $\bigoplus_{\ell_I\in[1,\mu_I]}\cO(U_{\wh I})(*D)\lcr\hb\rcr x^{\mu-\ell_I}$ (they have infinite rank as such), each of which equipped with a logarithmic $\cO(U_{\wh{1,I}})$-connection given by the formulas above, with $i\notin I$, and moreover equipped with $\cO(U_{\wh{1,I}})$-linear endomorphisms $x_i\wt\nabla^f_{x_i}$ ($i\in I$). It is enough to prove the comparison result of Lemma~\ref{lem:comparisonlcr} for each pair of graded complexes $\gr_p^F$. We therefore fix $I$ with $\#I\geq1$, and $\ell_I\in[1,\mu_I]$.

Let us fix $i_o\in I$. We define a logarithmic connection $\ov\nabla$ on $\cO(U_{\wh I})(*D)$ by setting
\begin{align*}
x_i\ov\nabla_{x_i}&=x_i\partial_{x_i}+\ell_{i_o}\frac{\mu_i}{\mu_{i_o}}\quad(i\geq1,\;i\notin I)\\
y_j\ov\nabla_{y_j}&=y_j\partial_{y_j}+a_{y_j}+\ell_{i_o}\frac{\nu_j}{\mu_{i_o}},\\
z_k\ov\nabla_{z_k}&=z_k\partial_{z_k}+a_{z_k}.
\end{align*}
Then the $(m-1)+n+p$-cube that we are considering has vertices equal to $\cO(U_{\wh I})(*D)\lcr\hb\rcr$ (\resp ...) and arrows (indexed by $i\neq1,j,k$) given by
\begin{align*}
&x_1\ov\nabla_{x_1}+(\ell_i-\ell_{i_o})\frac{\mu_i}{\mu_{i_o}}\quad\text{if }i\in I,\\
&x_i\ov\nabla_{x_i}-\frac{\mu_i}{\mu_1}x_1\ov\nabla_{x_1}\quad\text{if }i\geq2,\;i\notin I,\\
&y_j\ov\nabla_{y_j}-\frac{\nu_j}{\mu_1}x_1\ov\nabla_{x_1},\\
&z_k\ov\nabla_{z_k}.
\end{align*}
or equivalently, by a triangular change of basis,
\[
x_1\ov\nabla_{x_1}\ (i=i_o),\quad (\ell_i-\ell_{i_o})\frac{\mu_i}{\mu_{i_o}}\ (i\in I\moins\{i_o\}),\quad x_i\ov\nabla_{x_i}\ (i\geq2,\;i\notin I),\quad y_j\ov\nabla_{y_j},\quad z_k\ov\nabla_{z_k}.
\]
If $\ell_i\neq\ell_{i_o}$ for some $i\in I\moins\{i_o\}$, then the corresponding complexes are both quasi-isomorphic to $0$. Otherwise, we are reduced to proving the comparison for the complexes corresponding to the cubes having vertices $\cO(U_{\wh I})(*D)\lcr\hb\rcr$ (\resp ...) and arrows
\[
x_1\ov\nabla_{x_1}\ (i=i_o),\quad x_i\ov\nabla_{x_i}\ (i\geq2,\;i\notin I),\quad y_j\ov\nabla_{y_j},\quad z_k\ov\nabla_{z_k}.
\]
We can therefore apply Lemma \ref{lem:comparaison}\eqref{lem:comparaison1} below.
\end{proof}

\begin{lemme}\label{lem:comparaison}
Let $\nabla=d+b$ be a flat $(D_x\cup D)$-logarithmic connection on $\cO_Y$, where
\[
b=\sum_ib_{x_i}\frac{dx_i}{x_i}+\sum_jb_{y_j}\frac{dy_j}{y_j}+\sum_kb_{z_k}\frac{dz_k}{z_k},\qquad b_{x_i},b_{y_j},b_{z_k}\in\CC.
\]
Then,
\begin{enumerate}
\item\label{lem:comparaison1}
for $U$ as above, the natural inclusion of complexes
\[
\big([\cO(U)(*D)\otimes\Omega(U)^\cbbullet(\log D_x)]\lcr\hb\rcr,\nabla\big)\to\big(\Omega(U\moins D)^\cbbullet(\log D_x)\lcr\hb\rcr,\nabla\big)
\]
is a quasi-isomorphism;
\item\label{lem:comparaison2}
the natural morphism of complexes
\[
\big([\cO_Y(*D)\otimes\Omega_Y^\cbbullet(\log D_x)]\lcr\hb\rcr,\nabla\big)\to\bR j_*\big(\Omega_X^\cbbullet(\log D_x)\lcr\hb\rcr,\nabla\big)
\]
is a quasi-isomorphism;
\item\label{lem:comparaison3}
the natural morphism of complexes
\[
\big([\cO_Y(*D)\otimes\Omega_Y^\cbbullet(\log D_x)]\lpr\hb\rpr,\nabla\big)\to\bR j_*\big(\Omega_X^\cbbullet(\log D_x)\lpr\hb\rpr,\nabla\big)
\]
is a quasi-isomorphism.
\end{enumerate}
\end{lemme}

\begin{proof}
We first prove \eqref{lem:comparaison1} and \eqref{lem:comparaison2} without $\lcr\hb\rcr$. Let us start with \eqref{lem:comparaison1}. We realize both complexes as the simple complexes associated to the $(m+n+p)$-dimensional cubes with vertices $\cO(U)(*D)$ (\resp $\cO(U\moins D)$) and arrows $x_i\partial_{x_i}+b_{x_i}$, $y_j\partial_{y_j}+b_{y_j}$ and $z_k\partial_{z_k}+b_{z_k}$. If one of the $b_{y_j},b_{z_k}$ is not an integer or if one of the $b_{x_i}$ is not a nonpositive integer, then each complex is quasi-isomorphic to zero. Otherwise, one easily argues by induction on $\dim Y$. The proof of \eqref{lem:comparaison2}  is similar.

Now, \eqref{lem:comparaison1} follows from the comparison for the coefficients of each $u^k$. For \eqref{lem:comparaison2}, we argue as for \eqref{eq:limU} by using the property that $H^k(U\moins D,\Omega_X^\ell(\log D_x)\lcr\hb\rcr)=0$ for $U$ Stein and $k>0$, as already recalled in \S\ref{subsec:prelim}.

Finally, \eqref{lem:comparaison3} reduces to the commutation of $\bR j_*$ and $\CC\lpr\hb\rpr\otimes_{\CC\lcr\hb\rcr}$. Let us notice that, for each $k$, the natural morphism
\[
\cH^k\big(\Omega_X^\cbbullet(\log D_x)\lcr\hb\rcr,\nabla\big)\to\cH^k\big(\Omega_X^\cbbullet(\log D_x),\nabla\big)\lcr\hb\rcr
\]
is an isomorphism, by applying \cite[Th\ptbl4.5]{Hartshorne75} to a fundamental system of Stein open neighbourhoods of each point of $X$. It follows that $\big(\Omega_X^\cbbullet(\log D_x)\lcr\hb\rcr,\nabla\big)$ has locally constant cohomology on each stratum of the natural stratification of $(X,D_x)$. Since this stratification is the restriction to $X$ of a Whitney stratification of $(Y,X)$, we get~\eqref{lem:comparaison3} by applying Proposition \ref{prop:compdirlim}.
\end{proof}

\subsubsection*{End of the proof of Proposition \ref{prop:logcomparaison}}
From Lemma \ref{lem:comparisonlcr} we conclude that the natural morphism
\[
\CC\lpr\hb\rpr\otimes_{\CC\lcr\hb\rcr}K^\cbbullet_Y(*D)\lcr\hb\rcr=K^\cbbullet_Y(*D)\lpr\hb\rpr\to\CC\lpr\hb\rpr\otimes_{\CC\lcr\hb\rcr}\bR j_*(K^\cbbullet_X\lcr\hb\rcr)
\]
is a quasi-isomorphism. Proposition \ref{prop:logcomparaison} now follows from Lemma \ref{lem:Rjstar} below.
\end{proof}

\begin{lemme}\label{lem:Rjstar}
The natural morphism
\[
\CC\lpr\hb\rpr\otimes_{\CC\lcr\hb\rcr}\bR j_*(K^\cbbullet_X\lcr\hb\rcr)\to\bR j_*(K^\cbbullet_X\lpr\hb\rpr)=\bR j_*\big(\CC\lpr\hb\rpr\otimes_{\CC\lcr\hb\rcr}K^\cbbullet_X\big)
\]
is a quasi-isomorphism.
\end{lemme}

\begin{proof}
We will use Proposition \ref{prop:compdirlim} with respect to the natural stratification of $D\cup D_x$. We first replace $K^\cbbullet_X\lcr\hb\rcr$ with $\coker \hb x_1\nabla^f_{x_1}$ as above, equipped with the induced differentials $x_i\wt\nabla^f_{x_i}$, etc., and whose terms are $\cO_{X_{\wh 1}}\lcr\hb\rcr$-modules (the notation is similar to that used in the proof of Lemma \ref{lem:comparisonlcr}): indeed, it is easy to check, as we already did above, that $\hb x_1\nabla^f_{x_1}$ is injective on each term of the complex $K^\cbbullet_X\lcr\hb\rcr$. According to Proposition \ref{prop:compdirlim}, the lemma now follows from Lemma \ref{lem:constructibilite} below.
\end{proof}

\begin{lemme}\label{lem:constructibilite}
There exists a finite filtration $F^\cbbullet$ of the complex $\coker \hb x_1\nabla^f_{x_1}$ such that each graded complex $\gr^p_F\coker \hb x_1\nabla^f_{x_1}$ has locally constant cohomology on each stratum $Y_\alpha$ contained in $X$.
\end{lemme}

\begin{proof}
We consider the increasing sequence of ideals:
\[
\cI^{(0)}=0\subset\cI^{(1)}\subset\cdots\subset\cI^{(m-1)}\subset\cI^{(m)}=\cO_{X_{\wh 1}}\lcr\hb\rcr,
\]
with
\[
\cI^{(p)}=\sum_{\substack{I\subset\{2,\dots,m\}\\ \#I=m-p}}\cO_{X_{\wh 1}}\lcr\hb\rcr\cdot x^I,\qquad(p\geq1)
\]
where we set $x^I=\prod_{i\in I}x_i$. Given a sheaf $\cF$ of $\cO_{X_{\wh 1}}\lcr\hb\rcr$-modules, we denote by $F^p\cF$ the subsheaf of local sections annihilated by some power of $\cI^{(p)}$. This forms a decreasing filtration of $\cF$. Due to the logarithmic form of the differentials, we obtain a filtration $F^\cbbullet\coker \hb x_1\nabla^f_{x_1}$ of the complex $\coker \hb x_1\nabla^f_{x_1}$.

We will show the assertion by an explicit local computation of the graded complex on $X$. As above, we will neglect the coordinates $t_\ell$. Let $I^o\subset\{1,\dots,m\}$ and let us consider the stratum $X_{\wh{I^o}}$ defined by $x_i=0\iff i\in I^o$.

Let us first notice that, if $I^o=\emptyset$, then $\hb x_1\nabla^f_{x_1}$ is onto, hence the complex $\coker\hb x_1\nabla^f_{x_1}$ is zero and there is nothing to prove. Hence we will assume that $I^o\neq\emptyset$.

If $1\in I^o$, we set $I^{\prime o}=I^o\moins\{1\}$ and we decompose $\cO_X\lcr\hb\rcr_{|X_{\wh I^o}}$ in a way similar to \eqref{eq:decomp}: for each open polydisc centered at a point $(x^o,y^o,z^o)$ in $X_{\wh I^o}$, let $V$ denote a polydisc in $X$ centered at $(x^o,y^o,z^o)$, which is small enough so that $x_i\neq0$ all over~$V$ for each $i\notin I^o$; then we decompose $\Gamma(V,\cO_X\lcr\hb\rcr])$ as in \eqref{eq:decomp}, by replacing~$U$ with~$V$ and we impose $I\subset I^{\prime o}$. Since $H^1(V,\cO_X\lcr\hb\rcr)=0$, we have termwise $\Gamma(V,\coker \hb x_1\nabla^f_{x_1})=\coker\Gamma(V,\hb x_1\nabla^f_{x_1})$.

The argument is then analogous to that of the proof of Lemma \ref{lem:comparisonlcr}: the complex $\Gamma(V,\coker \hb x_1\nabla^f_{x_1})$ is identified with the Koszul complex associated to the cube whose vertices are given by the (modified) second line of \eqref{eq:decomp}. The filtration $F^\cbbullet$ is identified with that given by $\#I$, and the graded complexes are shown to be locally constant by 
using Lemma \ref{lem:comparaison}.

If $1\notin I^o$, the argument is similar, except that the first term of the second line in \eqref{eq:decomp} does not show up.
\end{proof}

\begin{proof}[End of the proof of Proposition \ref{prop:comparaison} in the normal crossing case]
We consider the partial weight filtration $W_\bbullet\Omega_Y^\cbbullet(*D)(\log D_x)$ with respect to the divisor $D_x$. Proposition \ref{prop:comparaison} concerns $W_0$, while Proposition \ref{prop:logcomparaison} gives the result for~$W_m$. We argue as in the proof of Lemma \ref{lem:formel}. It remains to show the result on $\gr_\ell^W$ for each $\ell\geq1$. One notices then that $df\hbm$ induces zero on such a quotient, and Proposition \ref{prop:comparaison} for $\gr_\ell^W$ reduces to the comparison result of Deligne \cite{Deligne70} as in Lemma \ref{lem:comparaison}\eqref{lem:comparaison3}.
\end{proof}

\backmatter
\def\ieme{rd\xspace}
\newcommand{\SortNoop}[1]{}\def\cprime{$'$}
\providecommand{\bysame}{\leavevmode ---\ }
\providecommand{\og}{``}
\providecommand{\fg}{''}
\providecommand{\smfandname}{\&}
\providecommand{\smfedsname}{\'eds.}
\providecommand{\smfedname}{\'ed.}
\providecommand{\smfmastersthesisname}{M\'emoire}
\providecommand{\smfphdthesisname}{Th\`ese}

\end{document}